\documentclass[12pt,a4paper,leqno]{amsart}
\usepackage{amssymb} 
\usepackage{mathrsfs} 
\usepackage{newtxtext}
\usepackage[varg]{newtxmath}
\usepackage[shortlabels]{enumitem} 
\usepackage{bm}
\usepackage[colorlinks=false]{hyperref}
\usepackage{graphicx, xcolor}

\usepackage{constants}
\newtheorem{theorem}{Theorem}[section]

\newtheorem{lemma}[theorem]{Lemma}

\newtheorem{remark}[theorem]{Remark}

\newtheorem{prop}[theorem]{Proposition}
\newtheorem{lem}[theorem]{Lemma}

\DeclareMathOperator{\Div}{div}
\DeclareBoldMathCommand{\bu}{u}
\DeclareBoldMathCommand{\bv}{v}
\DeclareBoldMathCommand{\bnu}{\nu}
\newcommand{\unknown}{\rho}
\let\unk\unknown
\newcommand{\F}{\mathcal{F}}
\newcommand{\D}{\mathcal{D}}

\newcommand{\R}{\mathbb{R}}
\newcommand{\dcenter}{d_{\text{center}}}
\newcommand{\dcoeff}{d_{\text{coeff}}}
\newcommand{\dosc}{d_{\text{osc}}}

\numberwithin{equation}{section}
\usepackage{mathtools}
\mathtoolsset{showonlyrefs=true}

\usepackage{amsrefs}

\title{Long-time asymptotics of nonlinear Fokker-Planck equations with inhomogeneous diffusion and natural boundary conditions}
\author{Kouta Araki}
\address{Department of Mathematics, Graduate School of Science and Technology, Nihon University, Tokyo 101-8308 JAPAN}
\email{csku24001@g.nihon-u.ac.jp}

\author{Masashi Mizuno}
\address{Department of Mathematics, College of Science and Technology,  Nihon University, Tokyo
101-8308 JAPAN}
\email{mizuno.masashi@nihon-u.ac.jp}
\allowdisplaybreaks[1]
\begin{document}

\begin{abstract}
    In this paper, we study the long-time asymptotics for nonlinear Fokker-Planck equations with spatial inhomogeneous diffusion, based on the law of energy dissipation. First, we derive a priori positive lower and upper bounds for solutions of the equations. Next, we give a sufficient condition for deriving the exponential decay of the dissipation function. Finally, we investigate whether the given sufficient condition is essential or not by numerical computation.
\end{abstract}

\maketitle

\section{Introduction}
\label{Inhomogeneous}

\subsection{Fokker-Planck equations with inhomogeneous diffusion}
Let $\Omega \subset \R^n$ be a bounded convex domain with smooth boundary in the $n$-dimensional Euclidean space, $\bnu$ be the outer unit normal vector on $\partial \Omega$. 
We consider the following initial-boundary value problem for the nonlinear Fokker-Planck equation:
\begin{equation}
    \label{eq:Nonlinear-Fokker-Planck}
    \tag{NFP}
    \left\{
    \begin{aligned}
        \frac{\partial \unknown}{\partial t}-\Div(\unknown\nabla (d(x)\log\unknown+\phi(x))) &=0, \quad && x \in \Omega, \quad t>0, 
        \\
        \unknown(0,x)&=\unknown_0(x) ,\quad &&x\in \Omega ,
        \\
        \unknown\nabla (d(x)\log\unknown+\phi(x))\cdot\bnu&=0,\quad &&x \in \partial \Omega,\quad t>0. 
    \end{aligned}
    \right.
\end{equation}
Here $d: \Omega \rightarrow \mathbb{R}$ is a given positive function, $\phi: \Omega \rightarrow \mathbb{R}$ is a given function, $\unk_0: \Omega \rightarrow \mathbb{R}$ is given initial data, and $\unk: \Omega \times [0, \infty) \rightarrow \mathbb{R}$ is an unknown function. \eqref{eq:Nonlinear-Fokker-Planck} 
was derived by \cites{MR4526584,MR4506846,arXiv:epshteyn2024longtimeasymptoticbehaviornonlinear} as a mathematical model to investigate grain boundary motion with absolute temperature. We briefly explain how to deduce \eqref{eq:Nonlinear-Fokker-Planck} from the law of energy dissipation.

Let us define the free energy of density function $\unk$ by
\begin{equation}
    \F[\unk](t) := \int_{\Omega} (d(x)(\log \unk - 1) + \phi(x))\unk \, dx.
\end{equation}
We impose the conservation law of mass, that is,
\begin{equation}
    \label{eq:conservation_mass}
    \unk_t + \Div(\unk \bu) = 0,
\end{equation}
for a velocity vector field $\bu$, and the law of energy dissipation of the form
\begin{equation}
    \frac{d}{dt}\F[\unk](t) = -\int_{\Omega} |\bu(x,t)|^2 \unk(x,t) \, dx =: -\D[\unk](t).
\end{equation}
Now, computing the rate of the free energy together with the conservation law, we have
\begin{align}
    \frac{d}{dt}\F[\unk](t) 
    &= 
    \int_{\Omega} (d(x)\log \unk + \phi(x))\unk_t \, dx \\
    &= -\int_{\Omega} (d(x)\log \unk + \phi(x))\text{div}(\unk \bu) \, dx.
\end{align}
Imposing the natural boundary condition
\begin{equation}
    \unk \bu \cdot \bnu|_{\partial\Omega} = 0, \quad t > 0,
\end{equation}
we have
\begin{equation}
    -\int_{\Omega} (d(x)\log \unk + \phi(x))\cdot \Div(\unk \bu) \, dx = \int_{\Omega} \nabla(d(x)\log \unk + \phi(x)) \cdot \bu\unk \, dx.
\end{equation}
To ensure the law of energy dissipation, we impose
\begin{equation}
    \label{eq:1.velocty_vector}
    \bu= -\nabla(d(x)\log \unk + \phi(x)),
\end{equation}
and we have \eqref{eq:Nonlinear-Fokker-Planck} by plugging the velocity vector field \eqref{eq:1.velocty_vector} into the conservation law of mass \eqref{eq:conservation_mass}.

In this paper, we study the long-time behavior of \eqref{eq:Nonlinear-Fokker-Planck} from the point of view of the law of energy dissipation. This argument can be seen in \cite{MR1833000} or \cite{MR4526584} to derive mathematical models of the grain boundary motion. In \cite{MR4526584}, $\unk$ represented the density of the grain boundaries in terms of the misorientations and the triple junctions, $\phi$ represented the grain boundary energy, and $d$ represented a fluctuation parameter.

Since $d$ is a fluctuation parameter, it would depend on the absolute temperature, but from the stochastic model, it is difficult to derive the law of energy dissipation. In contrast, the authors in \cite{MR4506846} derived the model \eqref{eq:Nonlinear-Fokker-Planck} from the law of energy dissipation to ensure it and included the spatially inhomogeneous diffusion.

When $d$ is a positive constant, this problem is related to the law of entropy dissipation; for instance, one can think of $\F[\unk](t)$ as the entropy of \eqref{eq:Nonlinear-Fokker-Planck}. If $\phi$ is strictly convex and $\F[\unk_0]$ is finite, then we can deduce the exponential decay of $\D[\unk](t)$ and derive the asymptotic convergence of $\unk$ as $t \rightarrow \infty$ in $L^1$ space \cites{MR1812873, MR1639292, MR1842428}.

On the other hand, when $d$ is a positive function, \cite{MR4506846} and \cite{arXiv:epshteyn2024longtimeasymptoticbehaviornonlinear} studied the exponential decay of $\D[\unk](t)$ for the periodic boundary value problem of \eqref{eq:Nonlinear-Fokker-Planck}, not the natural boundary value problem. From the point of view of the grain boundary motion, in mathematical analysis, periodic boundary conditions are easy to handle the boundary terms, but in the modelling, the natural boundary conditions are important to study the phenomenon. Thus, our setting is closer to the natural modelling. However, we need to investigate the boundary integral mathematically from the integration by parts. In this paper, we observe and estimate each boundary integral.

\subsection{Main results}

Here we explain the main results of this paper. Throughout this paper, we assume that $d$ is a given $C^2$ function on $\overline{\Omega}$ such that there exist positive constants $d_{\max}, \Cl{const:d>c}>0$ such that 
\begin{equation}
\label{assume:d>c}
   d_{\max}\ge d(x)\ge \Cr{const:d>c}=:d_{\min}
\end{equation}
for $x\in\Omega$. Next, we assume that $\phi$ is a given $C^2$ strictly convex function on $\overline{\Omega}$, that is, $\nabla^2\phi\ge \lambda I$ for some $\lambda>0$, where $I$ is the identity matrix. And we assume that $\unknown_0$ is a given positive $C^2$ probability density function on $\overline{\Omega}$, that is, 
\begin{equation}
\label{as:unknown_0_is_PDF}
    \int_\Omega \unknown_0(x)\,dx=1.
\end{equation}
Note that any classical solution $\unk$ of \eqref{eq:Nonlinear-Fokker-Planck} satisfies
\begin{equation}
\label{eq:f_pdf}
    \int_\Omega \unknown\,dx=\int_\Omega \unknown_0\,dx
\end{equation}
by taking integration of the first equation of \eqref{eq:Nonlinear-Fokker-Planck}, using the divergence theorem, and applying the boundary condition of \eqref{eq:Nonlinear-Fokker-Planck}.

We first give a sufficient condition that $\unk(t,x)$ is positive for any $x\in\Omega$ and $t>0$.

\begin{theorem}
\label{thm:apriori_estimate}
    Let $\unknown$ be a classical solution of \eqref{eq:Nonlinear-Fokker-Planck}. Then, there exist positive constants $\Cl{const:f>c}, \Cl{const:f<c}>0$ 
    depending only on $d$, $\phi$, and $\unk_0$
    such that 
    \begin{equation}
    \label{eq:apriori}
        \Cr{const:f>c} \le \unknown \le \Cr{const:f<c}.
    \end{equation}
\end{theorem}

\newcommand{\unkeq}{\unknown_{\text{eq}}}

Next, we discuss long-time asymptotics of global-in-time solutions of \eqref{eq:Nonlinear-Fokker-Planck}. 
Since the equilibrium state $\unkeq$ of \eqref{eq:Nonlinear-Fokker-Planck} satisfies $\bu=\bm{0}$ hence
\begin{equation}
    d(x)\log\unkeq +\phi(x)=\Cr{const:1.equilibrium_const},
\end{equation}
where $\Cl{const:1.equilibrium_const}$ is a constant determined by the relation
\begin{equation}
    \int_\Omega \unkeq(x)\,dx =1.
\end{equation}
Thus, we can write the equilibrium state as
\begin{equation}
    \unkeq(x)
    =
    \exp\left(
    \frac{\Cr{const:1.equilibrium_const}-\phi(x)}{d(x)}
    \right).
\end{equation}
Note that if $d$ is a constant, then the equilibrium state can be expressed as
\begin{equation}
    \unkeq(x)
    =
    \Cl{const:1.Boltzmann_distribution}
    \exp\left(
    \frac{-\phi(x)}{d}
    \right),
\end{equation}
which is called the Boltzmann distribution. The next theorem shows that the probability density function associated with  \eqref{eq:Nonlinear-Fokker-Planck} converges to the equilibrium state from the perspective of the law of energy dissipation if $\nabla \log d(x)$ is sufficiently small.

\begin{theorem}
\label{thm:exp_decay}
    Let $n=1,2,3$. Let $\unknown_0$ be a strictly positive bounded given initial data.
    Then there are constants $\Cl{const:nabla_log_d<c}, \Cl{const:D[unknown]<C}, \Cl{const:D[unknown]_exp_decay}>0$ such that if two conditions 
    \begin{equation}
        \|\nabla \log d\|_\infty<\Cr{const:nabla_log_d<c},
        \quad \D[\unknown_0]\le \Cr{const:D[unknown]<C}
    \end{equation}
    hold, then for the associated classical solution $\unknown$ of \eqref{eq:Nonlinear-Fokker-Planck}, we have 
    \begin{equation}
        \label{eq:1.exp_dacay_Dissipation}
        \D[\unknown](t)
        \le 
        \Cr{const:D[unknown]_exp_decay}e^{-\lambda t}.
    \end{equation}
\end{theorem}

\begin{remark}
Note from $\nabla \log d(x) ={\nabla d(x)}/{d(x)}$ that the smallness of $\nabla \log d(x)$ implies either the smallness of $\nabla d(x)$ or the largeness of $d(x)$. In \cites{araki2025longtimebehaviorfreeenergy,arXiv:epshteyn2024longtimeasymptoticbehaviornonlinear, MR4506846}, it was pointed out that the largeness of $d(x)$, i.e., stronger diffusion, implies the exponential decay of the dissipation function. Furthermore, the smallness of $\nabla d(x)$ means that inhomogeneous diffusion is close to homogeneous diffusion. Thus, the assumption of smallness on $\nabla \log d(x)$ implies both the strength of diffusion and its closeness to homogeneous diffusion.
\end{remark}

\begin{remark}
We give a remark on the existence of global-in-time solutions of \eqref{eq:Nonlinear-Fokker-Planck}. In Theorem \ref{thm:exp_decay}, we assume existence of a global-in-time solution. The existence of local-in-time classical solutions for sufficiently smooth initial data was shown by \cite{MR4547562}, however it is not clear whether a global-in-time solution exists or not. On the other hand, we note that \cite{MR4976469} showed the existence of global-in-time solutions for the non-linear Fokker-Planck equation subjected to the ``periodic'' boundary condition.
\end{remark}

To prove Theorem \ref{thm:apriori_estimate}, we use the change of variable 
\begin{equation}
    \label{mu}
    \mu:=d(x)\log \unknown +\phi(x).
\end{equation}
Then, new variable $\mu$ satisfies the homogeneous Neumann boundary condition and
\begin{equation}
    \label{eq:1.equation_potential_function}
    \mu_t
    -
    d(x)\Delta\mu
    +
    (\log\unknown\nabla d(x)+\nabla\phi(x))
    \cdot
    \nabla \mu
    -
    |\nabla\mu|^2
    =
    0
\end{equation}
hence $\mu$ is a sub/super-solution of some linear parabolic equation. Using the maximum principle for the linear parabolic equation, we obtain \eqref{eq:apriori}.

Next, we briefly explain how to derive exponential decay \eqref{eq:1.exp_dacay_Dissipation}. We first compute the time derivative of the dissipation function $\D[\unk](t)$. Then, we obtain
\begin{equation}
    \label{eq:1.rate_dissipation}
    \frac{d}{dt}\D[\unk](t)
    \leq
    -2
    \lambda
    \int_\Omega |\nabla\mu|^2\unknown\, dx
    -
    2\int_\Omega d(x) |\nabla^2\mu|^2\unknown\, dx
    +
    \int_\Omega 
    G(\nabla d(x))
    \unk
    \,dx
\end{equation}
for some function $G$. The precise computation can be seen in \eqref{eq:exp_decay.rate_dissipation}. The function $G(\nabla d)$ can be estimated by 
\begin{equation}
    \label{eq:1.dissipation_inhomogeneous}
    G(\nabla d)
    \sim
    |\nabla d||\nabla\mu||\nabla^2 \mu|
    +
    |\nabla d||\nabla\mu|^3.
\end{equation}
Using 
\begin{equation}
    |\nabla d||\nabla\mu||\nabla^2 \mu|
    \leq
    \frac{1}{2}d(x)|\nabla^2 \mu|^2
    +
    \frac{1}{2}d(x)|\nabla \log d(x)|^2|\nabla\mu|,
\end{equation}
the first term of \eqref{eq:1.dissipation_inhomogeneous} can be controlled by the first and second term in the right-hand side of \eqref{eq:1.rate_dissipation} or smallness of $\|\nabla \log d\|_\infty$. 
To control the second term of \eqref{eq:1.dissipation_inhomogeneous}, we use the Sobolev inequality and interpolation. This is the reason why we assume the dimensional assumption $n=1,2,3$.

Finally, we numerically examine whether the smallness of $\|\nabla \log d\|_\infty$ is essential or not. To do this, we first consider \eqref{eq:1.equation_potential_function}, equation of $\mu$, to connect with the forward Euler scheme for nonlinear parabolic equations subjected to the Neumann boundary condition. Next, we calculate the decay rate of  the dissipation function $\D[\unk](t)$.

\subsection{Organization of the paper and notation}

In Section \ref{sec:Maximam_principle}, we prove Theorem \ref{thm:apriori_estimate}, which provides \textit{a priori} bounds for solutions to \eqref{eq:Nonlinear-Fokker-Planck}. These estimates are required to establish the asymptotic behavior of the dissipation function. In Section \ref{sec:proof_of_Main_theorem}, we present Theorem \ref{thm:exp_decay}, which establishes the exponential decay of the dissipation function. In Section \ref{sec:Numerical_Analysis}, we explain the numerical computation of \eqref{eq:Nonlinear-Fokker-Planck} and investigate the interaction between the smallness of $\nabla \log d$ and the decay rate of the dissipation function.

Let $\Omega\subset\R^n$ be an open set and let $f\colon\Omega\rightarrow\R$ be
a sufficiently smooth function $f:\Omega\rightarrow\R$. We denote the
gradient of $f$ as
\begin{equation}
 \nabla f
  :=
  \left(
   \frac{\partial f}{\partial x_1}, \frac{\partial f}{\partial x_2}, \dots, \frac{\partial f}{\partial x_n}
   \right).
\end{equation}
We denote the Hesse matrix of $f$ as
\begin{equation}
 \nabla^2f:=
  \begin{pmatrix}
         \frac{\partial^2 f}{\partial x_1^2 } &\cdots &\frac{\partial^2 f}{\partial x_1\partial x_i}& \cdots&\frac{\partial^2 f}{\partial x_1 \partial x_n}\\
         \vdots & \ddots& & &\vdots \\
         \frac{\partial^2 f}{\partial x_i\partial x_1}&\cdots&\frac{\partial^2 f}{\partial x_i^2}&\cdots&\frac{\partial^2 f}{\partial x_i \partial x_n}\\
         \vdots& & &\ddots&\vdots\\
         \frac{\partial^2 f}{\partial x_n\partial x_1}&\cdots
         &\frac{\partial^2 f}{\partial x_n\partial x_i}&\cdots&\frac{\partial^2 f}{\partial x_n^2}
  \end{pmatrix}.
\end{equation}
The Laplacian of $f$ is denoted as
\begin{equation}
 \Delta f
  :=
  \sum_{i=1}^n \frac{\partial^2 f}{\partial x_i^2}.
\end{equation}
For $n$-dimensional symmetric matrices $X, Y$,  we define $X\le Y$ to be
the case that for all $\xi \in \R^n$
\begin{equation*}
 X \xi \cdot \xi \le Y \xi \cdot\xi.
\end{equation*}
We denote $I$ the $n$-dimensional identity matrix. Thus,  for
$n$-dimensional symmetric matrix $X$,  $cI\leq X$ for some $c\in\R$ means
that the eigenvalue of $X$ is greater than or equal to $c$.

\section{proof of Theorem\ref{thm:apriori_estimate}}
\label{sec:Maximam_principle}
The proof of Theorem~\ref{thm:apriori_estimate} follows the maximum principle for the \eqref{eq:Nonlinear-Fokker-Planck}.
Namely, \eqref{eq:Nonlinear-Fokker-Planck} can be written in the non-divergence form as follows.
\begin{equation*}
     \begin{split}
        \unknown_t&-d(x)\Delta\unknown-(\nabla d(x)+\log\unknown\nabla d(x)+\nabla d(x)+\nabla \phi(x))\cdot\nabla\unknown
        \\
        &-\unknown(\log\unknown\Delta d(x)+\Delta\phi(x))=0.
     \end{split}
\end{equation*}
Since the above equation contains zeroth-order derivative terms of $\unknown$, the method to prove the maximum principle cannot be applicable directly. 
Thus, we consider change of variable from $\unknown$ to $\mu=d(x)\log \unknown +\phi(x)$. We first derive the equation of the potential $\mu$. 
\begin{lemma}
\label{lem:mu's_equation}
    Let $\unknown$ be a classical solution of \eqref{eq:Nonlinear-Fokker-Planck} and let $\mu$ be defined by \eqref{mu}.
    Then, we obtain that 
    \begin{equation}
    \label{eq:mu's_equation}
         \mu_t-d(x)\Delta\mu+(\log\unknown\nabla d(x)+\nabla\phi(x))\cdot\nabla \mu=|\nabla\mu|^2.
    \end{equation}
\end{lemma}
\begin{proof}
We directly calculate
\begin{equation*}
    \begin{split}
        \mu_t&=d(x)\frac{1}{\unknown}\unknown_t
        \\
        &=d(x)\frac{1}{\unknown}\Div(\unknown\nabla\mu)
        \\
        &=d(x)\Delta\mu+d(x)\nabla\log \unknown\cdot \nabla\mu.
    \end{split}
\end{equation*}
Since $\nabla\mu=d(x)\nabla\log \unknown+\log \unknown \nabla d(x)+\nabla\phi(x)$, we obtain \eqref{eq:mu's_equation}.

\end{proof}
For the classical solution $\unknown$ of \eqref{eq:Nonlinear-Fokker-Planck}, define a linear uniformly parabolic operator 
\begin{equation}
\label{def:linear-op}
    L[\mu]:=-d(x)\Delta\mu +(\log \unknown \nabla d(x)+\nabla \phi(x))\cdot \nabla\mu.
\end{equation}
Then, \eqref{eq:mu's_equation} is transformed into \begin{equation}
\label{eq:potential_equation}
\left\{
\begin{aligned}
    \mu_t+L[\mu]&=|\nabla\mu|^2, \quad &&x \in \Omega , t>0,\\
    \nabla\mu\cdot\bnu&=0,\quad &&x\in \partial \Omega, t>0.
\end{aligned}
\right.
\end{equation}
Since $L[\mu]$ is a linear uniformly parabolic operator, the maximum principle is applicable by \cites{MR219861,MR1625845}.
\begin{lemma}[\cites{MR219861,MR1625845}]
\label{lem:maximum_principle_general}
    Let $L$ be a uniformly parabolic operator defined in \eqref{def:linear-op}.
    \begin{enumerate}
        \item Let $w$ satisfy the uniformly parabolic differential inequality
        \begin{equation}
        \left\{
        \begin{aligned}
            w_t+L[w]&\ge0,\quad in &&\Omega \times (0,T), \\
            \nabla w\cdot\bnu&=0,\quad on &&\partial \Omega\times (0,T).
            \end{aligned}
            \right.
        \end{equation}
        Then we obtain 
        \begin{equation}
            \min_{y \in \overline{\Omega}}w(y,0)\le w(x,t), \quad (x,t)\in \overline{\Omega}\times(0,T).
        \end{equation}
        \item Let $w$ satisfy the uniformly parabolic differential inequality
        \begin{equation}
        \left\{
        \begin{aligned}
            w_t+L[w]&\le0,\quad in &&\Omega \times (0,T), \\
            \nabla w\cdot\bnu&=0,\quad on &&\partial \Omega\times (0,T).
            \end{aligned}
            \right.
        \end{equation}
        Then we obtain 
        \begin{equation}
             w(x,t) \le \max_{y \in \overline{\Omega}}w(y,0), \quad (x,t)\in \overline{\Omega}\times(0,T).
        \end{equation}
    \end{enumerate}

\end{lemma}

Next, we show the lower bounds of $\unknown$ by applying Lemma~\ref{lem:maximum_principle_general}.

\begin{lemma}
    Let $\unknown$ be a classical solution of \eqref{eq:Nonlinear-Fokker-Planck}. Then, we obtain 
    \begin{equation}
        \label{eq:2.LowerBounds_unknown}
        \Cr{const:f>c} \le \unknown,
    \end{equation}
    where $\Cr{const:f>c}$ is a positive constant depending only on $\unknown_0$, $d(x)$, and $\phi$.
\end{lemma}
\begin{proof}
From \eqref{eq:potential_equation}, we have for $x \in \Omega$ and $t>0$, 
\begin{equation}
    \mu_t+L[\mu]=|\nabla\mu|^2\ge0.
\end{equation}
Thus, from Lemma~\ref{lem:maximum_principle_general}, we obtain 
\begin{equation}
\begin{split}
    \min_{y \in \overline{\Omega}}\mu(y,0) 
    &\le \mu(x,t)
    \\
    &=d(x)\log\unknown(x,t)+\phi(x)
    \\
    & \le d(x)\log\unknown(x,t)+\max_{y\in\overline{\Omega}}\phi(y)
\end{split}
\end{equation}
for $x\in \Omega$ and $t>0$. Then, 
\begin{equation}
    \frac{\min_{y \in \overline{\Omega}}\mu(y,0)-\max_{y\in \overline{\Omega}}\phi(y)}{d(x)}\le \log\unknown(x,t)
\end{equation}
hence we have \eqref{eq:2.LowerBounds_unknown}.
\end{proof}

\begin{remark}
\label{remark:c_1<unknown<c_2}
    If $\min_{y \in \overline{\Omega}}\mu(y,0)\ge \max_{y \in \overline{\Omega}}\phi(y)$, then constant $\Cr{const:f>c}$ can be explicitly written as 
    \begin{equation}
        \Cr{const:f>c}=\exp\left(\frac{\min_{y \in \overline{\Omega}}(d(y)\log \unknown_0(y)+\phi(y))-\max_{y \in \overline{\Omega}}\phi(y)}{\max_{y \in \overline{\Omega}}d(y)}\right).
    \end{equation}
    Conversely, if $\min_{y \in \overline{\Omega}}\mu(y,0)\le \max_{y \in \overline{\Omega}}\phi(y)$, then constant $\Cr{const:f>c}$ can be written as 
    \begin{equation}
        \Cr{const:f>c}=\exp\left(\frac{\max_{y \in \overline{\Omega}}(d(y)\log \unknown_0(y)+\phi(y))-\min_{y \in \overline{\Omega}}\phi(y)}{\min_{y \in \overline{\Omega}}d(y)}\right).
    \end{equation}
    In either case, positive constant $\Cr{const:f>c}$ can be explicitly given by 
    $\min_{y \in \overline{\Omega}}\mu(y,0)$, $\max_{y \in \overline{\Omega}}\phi(y)$, $\max_{y \in \overline{\Omega}}d(y)$ and $\min_{y \in \overline{\Omega}}d(y)$.
\end{remark}

Next, we derive the upper bounds of $\unknown$. To do this, we set $\eta=e^{\beta\mu}$ for $\beta>0$. Then we compute
\begin{equation*}
        \eta_t=\beta e^{\beta \mu}\mu_t, \quad
        \nabla\eta=\beta e^{\beta \mu}\nabla\mu, \quad
        \Delta\eta= \beta^2 e^{\beta \mu}|\nabla\mu|^2+\beta e^{\beta\mu}\Delta \mu.
\end{equation*}
Plugging \eqref{eq:mu's_equation} into $\eta_t$, we obtain
\begin{equation}
    \begin{split}
        \eta_t&=\beta e^{\beta\mu}(d(x)\Delta \mu-(\log \unknown\nabla d(x) +\nabla\phi(x))\cdot \nabla\mu+|\nabla\mu|^2)
        \\
        &=d(x)(\Delta\eta-\beta^2 e^{\beta\mu}|\nabla\mu|^2)-(\log\unknown\nabla d(x)+\nabla\phi(x))\cdot\nabla\eta+\beta e^{\beta\mu}|\nabla\mu|^2
        \\
        &=d(x)\Delta\eta-(\log\unknown\nabla d(x)+\nabla\phi(x))\cdot\nabla\eta+\beta e^{\beta\mu}(1-d(x)\beta)|\nabla\mu|^2.
    \end{split}
\end{equation}
We take $\beta$ sufficiently large that $(1-d(x)\beta)<0$ for $x\in\Omega$. Then we have the following:
\begin{equation}
\label{eq:subsolution_mu}
    \eta_t-d(x)\Delta\eta+(\log\unknown\nabla d(x)+\nabla\phi(x))\cdot\nabla\eta\le 0.
\end{equation}
Since $\eta$ satisfies the differential inequality $\eta_t+L[\eta] \le 0$, we can apply the strong maximum principle to $\eta$. 
\begin{lemma}
    Let $\unknown$ be a classical solution of \eqref{eq:Nonlinear-Fokker-Planck}. Then, we obtain 
    \begin{equation}
    \label{eq:unknow<=C}
        \unknown\le \Cr{const:f<c},
    \end{equation}
    where $\Cr{const:f<c}$ is a positive constant depending only on $\unknown_0$, $d$ and $\phi$.
\end{lemma}

\begin{proof}
Let $\eta=e^{\beta\mu}$. Then $\eta$ satisfies
\begin{equation*}
    \eta_t+L[\eta] \le 0,\quad \text{in} \, \Omega\times(0,T),
\end{equation*}
and 
\begin{equation*}
    \nabla\eta\cdot\bnu=\beta e^{\beta\mu}\nabla\mu\cdot\bnu=0, \quad \text{on}\, \partial \Omega\times(0,T).
\end{equation*}
Hence we can apply Lemma \ref{lem:maximum_principle_general} and 
\begin{equation*}
    \eta(x,t)\le \max_{y \in \overline{\Omega}}\eta(y,0)
\end{equation*}
for $x \in \Omega$ and $t>0$. Thus, we have 
\begin{equation*}
    \mu(x,t)\le \max_{y\in \overline{\Omega}}\mu(y,0),
\end{equation*}
that is 
\begin{equation*}
    d(x)\log \unknown(x,t) \le \max_{y \in \Omega }\mu(y,0)-\min_{y \in \overline{\Omega}}\phi(y).
\end{equation*}
Therefore, \eqref{eq:unknow<=C} is derived.
\end{proof}
\begin{remark}
    As in a similar consideration in Remark \ref{remark:c_1<unknown<c_2}, a positive constant $\Cr{const:f<c}$ can be explicitly written by $\max_{y \in \overline{\Omega}}\mu(y,0)$, $\min_{y \in \overline{\Omega}}\phi(y)$, $\max_{y \in \overline{\Omega}}d(y)$, and $\min_{y \in \overline{\Omega}}d(y)$.
\end{remark}

\section{Proof of theorem \ref{thm:exp_decay}}
\label{sec:proof_of_Main_theorem}
The exponential decay of the dissipation function $\D[\unk](t)$ will be shown by evaluating the time derivative of $\D[\unk](t)$.
By direct computation, we have the following:
\begin{equation}
    \label{d^2F/dt^2_0}
    \begin{split}
        \frac{d}{dt}\D[\unknown](t)
        &=
        \frac{d}{dt}\left(\int_\Omega |\nabla\mu|^2\unknown\, dx\right)
        \\
        &=
        \int_\Omega |\nabla\mu|^2\unknown_t\, dx+2\int_\Omega (\nabla\mu\cdot\nabla\mu_t)\unknown\, dx.
    \end{split}
\end{equation}

We first compute the right-hand side of \eqref{d^2F/dt^2_0}.

\begin{lem}
    \label{lem:second-derivative-F}
    Let $\unknown$ be a classical solution of \eqref{eq:Nonlinear-Fokker-Planck} on $\overline{\Omega}\times [0, \infty)$. Then, 
    \begin{equation}
    \label{eq:second-derivative-F}
        \begin{split}
            \frac{d}{dt}\D[\unknown](t)
            &=
            -2\int_\Omega (\nabla\mu\cdot \nabla^2(d(x)\log\unknown)\nabla\mu)\unknown\, dx
            \\
            &\quad
            -2\int_\Omega (\nabla\mu\cdot\nabla^2\phi(x)\nabla\mu)\unknown\, dx
            \\
            &\quad
            -2\int_\Omega d(x)(\nabla\log \unknown\cdot\nabla\mu)^2\unknown\,dx 
            \\
            &\quad
            -4\int_\Omega d(x)(\nabla\log \unknown\cdot\nabla\mu)\Delta\mu \unknown\,dx
            \\
            &\quad
            -2\int_\Omega d(x)(\Delta\mu)^2\unknown\,dx.
        \end{split}            
    \end{equation}
\end{lem}
\begin{proof}
    Using the integration by parts and \eqref{eq:Nonlinear-Fokker-Planck}, the first term on the right hand side of \eqref{d^2F/dt^2_0} turns into
\begin{equation*}
   \int_\Omega |\nabla\mu|^2\unknown_t\, dx=\int_\Omega |\nabla\mu|^2\Div(\unknown\nabla\mu)\, dx
            =-\int_\Omega (\nabla(|\nabla\mu|^2) \cdot\nabla\mu)\unknown\, dx.
\end{equation*}
Next,  we compute $\nabla(|\nabla \mu|^2)\cdot\nabla\mu$. We denote $\nabla\mu=(\mu_{x_1}, \mu_{x_2}, \dots \mu_{x_n})$. Then,  by direct calculation, we obtain
    \begin{equation*}
        \begin{split}(\nabla(|\nabla\mu|^2)\cdot\nabla\mu)&=\sum_{i=1}^n\left(\sum_{j=1}^n\mu_{x_j}^2\right)_{x_i}\mu_{x_i} \\
            &=\sum_{i, j=1}^n 2\mu_{x_j}\mu_{x_jx_i}\mu_{x_i}\\
            &=2(\nabla\mu\cdot\nabla^2\mu\nabla\mu).
        \end{split}
    \end{equation*}
    Since $\mu= d(x) \log\unknown+\phi(x)$, the following holds:
    \begin{equation}
    \label{eq:|nablamu|^2unknown_t}
        \begin{split}
            \int_\Omega |\nabla\mu|^2\unknown_t\, dx&=-2\int_\Omega (\nabla\mu\cdot \nabla^2(d(x)\log\unknown)\nabla\mu)\unknown\, dx
            \\
            &\quad-2\int_\Omega (\nabla\mu\cdot\nabla^2\phi(x)\nabla\mu)\unknown\,dx.
        \end{split}
    \end{equation}
    
    We next compute the second term on the right hand side of \eqref{d^2F/dt^2_0}.
    Using \eqref{eq:Nonlinear-Fokker-Planck} and the definition of $\mu$, 
    \begin{equation}
    \label{eq:nablamu_t_1}
    \nabla\mu_t=\nabla(d(x)\unknown^{-1}\unknown_t)
    =\nabla(d(x)\unknown^{-1}\Div(\unknown\nabla\mu)).
    \end{equation}
    Using integration by parts together with the boundary condition of \eqref{eq:Nonlinear-Fokker-Planck},  we have
    \begin{equation*}
        \begin{split}
            \int_\Omega (\nabla\mu\cdot\nabla\mu_t)\unknown\, dx
             &=
             \int_\Omega (\nabla\mu\cdot\nabla(d(x)\unknown^{-1}\Div(\unknown\nabla\mu)))\unknown\, dx
             \\
            &=-\int_\Omega d(x)\unknown^{-1}(\Div(\unknown\nabla\mu))^2\, dx
            \\
            &=-\int_\Omega d(x)(\nabla \unknown\cdot\nabla\mu)^2\unknown^{-1}\, dx
            \\
            &\quad
            -2\int_\Omega d(x)(\nabla \unknown\cdot\nabla\mu)\Delta\mu \, dx
            \\
            &\quad
            -\int_\Omega d(x)(\Delta\mu)^2\unknown\, dx.
        \end{split}
    \end{equation*}
    Since $\nabla\log \unknown =\unknown^{-1}\nabla \unknown$, we obtain
    \begin{equation}
    \label{eq:nablamu_t_2}
        \begin{split}
            2\int_\Omega(\nabla\mu\cdot\nabla\mu_t)\unknown\,dx
            &=-2\int_\Omega d(x)(\nabla\log\unknown\cdot\nabla\mu)^2\unknown\,dx 
            \\
            &\quad-4\int_\Omega d(x)(\nabla\log\unknown\cdot\nabla\mu)\Delta\mu\unknown\,dx
            \\
            &\quad-2\int_\Omega d(x)(\Delta\mu)^2\unknown\,dx.
        \end{split}
    \end{equation}
    Plugging \eqref{eq:|nablamu|^2unknown_t} and \eqref{eq:nablamu_t_2} to \eqref{d^2F/dt^2_0}, we have \eqref{eq:second-derivative-F}

\end{proof}

We prepare the following lemma to estimate $\nabla^2 (d(x)\log\unknown)$ in \eqref{eq:second-derivative-F}
\begin{lem}
    \label{nabla^2(d(x)f^(alpha-1))}
    Let $\unknown$ be a classical solution of \eqref{eq:Nonlinear-Fokker-Planck} on $\overline{\Omega}\times [0, \infty)$. Then, 
    \begin{equation}
    \label{nabla^2(d(x)f^(alpha-1))1}
        \begin{split}
        &\quad
         \int_{\Omega}(\nabla\mu\cdot\nabla^2(d(x)\log\unknown)\nabla\mu)\unknown\, dx
        \\
        &=   
        -
        \int_\Omega
        \nabla(d(x)\log\unknown)\cdot
        \left(
        \nabla^2\mu\nabla\mu
        +
        \Delta\mu \nabla\mu
        +
        (\nabla\mu\cdot\nabla\log \unknown)\nabla\mu
        \right)
        \unknown\, dx. 
        \end{split}
    \end{equation}
\end{lem}

\begin{proof}
    We compute $\nabla^2(d(x)\log\unknown)$. We denote 
    \begin{equation}
        \nabla^2 (d(x)\log\unknown)=((d(x)\log\unknown)_{x_ix_j})_{i, j}. 
    \end{equation}
    Then,  by direct calculations, we obtain
    \begin{equation*}
        \begin{split}(\nabla\mu\cdot\nabla^2(d(x)\log\unknown)\nabla\mu)\unknown
            &=  \sum_{i=1}^n\left(\mu_{x_i}\left(\sum_{j=1}^n(d(x)\log\unknown)_{x_ix_j}\mu_{x_j}\right)\right)\unknown
            \\
            &=
            \sum_{i, j=1}^n\left(\mu_{x_i}(d(x)\log\unknown)_{x_i}\mu_{x_j}\unknown\right)_{x_j}
        \\
            &\quad
            -\sum_{i, j=1}^n\left((d(x)\log\unknown)_{x_i}(\mu_{x_i}\mu_{x_j}\unknown)_{x_j}\right).
        \end{split}
    \end{equation*}
    The first term of the right-hand side turns into
    \begin{equation*}
        \sum_{i, j=1}^n\left(\mu_{x_i}(d(x)\log\unknown)_{x_i}\mu_{x_j}\unknown\right)_{x_j}=\Div((\nabla\mu\cdot\nabla(d(x)\log\unknown)\unknown\nabla\mu).
    \end{equation*}
    By calculating the second term of the right-hand side,  we obtain
    \begin{equation*}
        \begin{split}
            &\qquad
            -\sum_{i, j=1}^n\left((d(x)\log\unknown)_{x_i}(\mu_{x_i}\mu_{x_j}\unknown)_{x_j}\right)
            \\
            &=
            -\sum_{i, j=1}^n\left((d(x)\log\unknown)_{x_i}
            \left(
            (
            \mu_{x_i x_j}\mu_{x_j})\unknown
            +
            \mu_{x_i}\mu_{x_jx_j}\unknown
            +
            \mu_{x_i}\mu_{x_j}\unknown_{x_j}
            \right)
            \right)\\
            &=
            -(\nabla(d(x)\log\unknown)\cdot\nabla^2\mu\nabla\mu)\unknown
            -(\nabla(d(x)\log\unknown)\cdot\nabla\mu)\Delta\mu \unknown
            \\
            &\quad
            -(\nabla(d(x)\log\unknown)\cdot\nabla\mu)(\nabla\mu\cdot\nabla\log \unknown)\unknown
            .
        \end{split}
    \end{equation*}
    Thus, we obtain
    \begin{equation}
    \label{eq:2.Hesse-Diffusion}
        \begin{split}
        &\quad
        (\nabla\mu\cdot\nabla^2(d(x)\log\unknown)\nabla\mu)\unknown
        \\
        &=
        \Div(\nabla\mu\cdot\nabla(d(x)\log\unknown)\unknown\nabla\mu)
        \\
        &\quad
        -\nabla(d(x)\log\unknown)\cdot
        \left(
        \unknown\nabla^2\mu\nabla\mu
        +
        \Delta\mu \unknown\nabla\mu
        +
        (\nabla\mu\cdot\nabla \log\unknown)\unknown\nabla\mu
        \right).
        \end{split}
    \end{equation}
    Therefore,  integrating on $\Omega$ of both sides of \eqref{eq:2.Hesse-Diffusion},  we have \eqref{nabla^2(d(x)f^(alpha-1))1}
    since the integral of the first term in the right-hand side of \eqref{eq:2.Hesse-Diffusion} vanishes by using the boundary condition of \eqref{eq:Nonlinear-Fokker-Planck} with the divergence theorem. 
\end{proof}

By the computation of the integrand of \eqref{nabla^2(d(x)f^(alpha-1))1}, we have
    \begin{equation}
        \begin{split}
            &\quad \nabla(d(x)\log\unknown)\cdot
        \left(
        \nabla^2\mu\nabla\mu
        +
        \Delta\mu \nabla\mu
        +
        (\nabla\mu\cdot\nabla\log \unknown)\nabla\mu
        \right)
        \\
        &=d(x)((\nabla \log\unknown\cdot\nabla^2\mu\nabla\mu)
        +(\nabla \log\unknown\cdot\nabla\mu)\Delta\mu
        +(\nabla \log\unknown\cdot\nabla\mu)^2)
        \\
        &\quad
        +\log\unknown\nabla d(x)\cdot(\nabla^2\mu\nabla\mu
        +\Delta\mu\nabla\mu
        +(\nabla \log\unknown\cdot\nabla\mu)\nabla\mu).
        \end{split}
    \end{equation}
Using \eqref{nabla^2(d(x)f^(alpha-1))1} in \eqref{eq:second-derivative-F},  the time derivative of $\D[\unk](t)$  can be expressed as follows.
    \begin{equation}
    \label{d^2F/dt^2_4}
        \begin{split}
            \frac{d}{dt}\D[\unknown](t)&=-2\int_\Omega (\nabla\mu\cdot\nabla^2\phi(x)\nabla\mu)\unknown\, dx
            \\
            &\quad
            +2\int_\Omega d(x)(\nabla \log\unknown\cdot\nabla^2\mu\nabla\mu)\unknown\, dx
            \\
            &\quad
            -2\int_\Omega d(x)(\nabla\log \unknown\cdot \nabla\mu)\Delta\mu \unknown\, dx
            \\
            &\quad
            -2\int_\Omega d(x) (\Delta\mu)^2\unknown\, dx
            \\
            &\quad
            +2\int_\Omega \log\unknown(\nabla d(x)\cdot\nabla^2\mu\nabla\mu)\unknown\, dx 
            \\
            &\quad
            +2\int_\Omega \log\unknown(\nabla d(x)\cdot\nabla\mu)\Delta\mu \unknown\, dx
            \\
            &\quad
            +2\int_\Omega \log \unknown(\nabla d(x)\cdot\nabla\mu)(\nabla\log \unknown\cdot\nabla\mu)\unknown\, dx.
        \end{split}
    \end{equation}
If $d$ is a constant,  \eqref{d^2F/dt^2_4} coincides with the previous
result about the entropy dissipation methods by
\cites{MR1853037, MR3497125},  that is, the last three terms of the
right-hand side in \eqref{d^2F/dt^2_4} appear in the effect of
inhomogeneity of the diffusion.

We proceed with the computation according to the entropy dissipation methods.
We consider the third term in the right-hand side of
\eqref{d^2F/dt^2_4}.
 
\begin{lem}
    \label{nablafnablamuDeltamu}
    Let $\unknown$ be a classical solution of \eqref{eq:Nonlinear-Fokker-Planck} on $\overline{\Omega}\times [0, \infty)$. Then
    \begin{equation}
    \label{D(nabla f nablamu)Deltamu f^alpha-1}
        \begin{split}
            &\quad
            2\int_\Omega d(x)(\nabla \log\unknown\cdot \nabla\mu)\Delta\mu\unknown\,dx
            \\
            &=
            -
            2\int_\Omega d(x)(\Delta\mu)^2\unknown\, dx
            -
            \int_\Omega d(x)\Delta (|\nabla\mu|^2)\unknown\, dx
            \\
            &\quad
            +
            2\int_\Omega d(x) |\nabla^2\mu|^2\unknown\, dx
            -
            2\int_\Omega (\nabla d(x)\cdot \nabla\mu)\Delta\mu \unknown\, dx.
        \end{split}
    \end{equation}
\end{lem}
\begin{proof}
    Since $d(x)\nabla (\log \unknown) \unknown=d(x)\nabla \unknown$, we obtain 
    \begin{equation}
        \label{eq:2.Pohozaev-method-3}
        d(x)(\nabla \log \unknown)\unknown
        =
        (d(x)\nabla\unk)
        =
        \nabla (d(x)\unknown)
        -
        \unk\nabla d(x).
    \end{equation}
    Next,  we compute $(\nabla(d(x)\unknown)\cdot\nabla\mu)\Delta\mu$. 
    Expressing the vector in component form,  we obtain
    \begin{equation}
        \label{eq:2.Pohozaev-method-1}
        (\nabla(d(x)\unknown)\cdot\nabla\mu)\Delta\mu
        =\sum_{i, j=1}^n(d(x)\unknown)_{x_i}\mu_{x_i}\mu_{x_jx_j}.
    \end{equation}
    We make a divergence form in the right-hand side of \eqref{eq:2.Pohozaev-method-1} as follows:
    \begin{equation}
        \label{eq:2.Pohozaev-method-2}
        (d(x)\unknown)_{x_i}\mu_{x_i}\mu_{x_jx_j}
        =
        (d(x)\unknown\mu_{x_i}\mu_{x_jx_j})_{x_i}
        -
        d(x)(\mu_{x_i}\mu_{x_jx_j})_{x_i}\unknown.
    \end{equation}
    Compute the second term of the right-hand side of \eqref{eq:2.Pohozaev-method-2}
    as
    \begin{equation}
        (\mu_{x_i}\mu_{x_jx_j})_{x_i}
        =
        \mu_{x_ix_i}\mu_{x_jx_j}
        +
        \mu_{x_i}\mu_{x_jx_jx_i}
        =
        \mu_{x_ix_i}\mu_{x_jx_j}
        +
        (\mu_{x_i}\mu_{x_ix_j})_{x_j}
        -
        \mu_{x_ix_j}\mu_{x_ix_j}.
    \end{equation}
    Note that $\mu_{x_i}\mu_{x_ix_j}=\frac{1}{2}(\mu_{x_i}^2)_{x_j}$. Thus,  we arrive at
    \begin{equation}
    \label{eq:nabla(d unknown^alpha)}
        \begin{split}
        &\quad 
        (\nabla(d(x)\unknown)\cdot\nabla\mu)\Delta\mu
        \\
        &=
        \sum_{i, j=1}^n
        \biggl(
         (d(x)\unknown\mu_{x_i}\mu_{x_jx_j})_{x_i}
         -d(x)\mu_{x_ix_i}\mu_{x_jx_j}\unknown
         \\
         &\qquad
         -\frac{d(x)}{2}(\mu_{x_i}^2)_{x_jx_j}\unknown
         +d(x)\mu_{x_ix_j}\mu_{x_ix_j}\unknown
          \biggr)
          \\
        &=
        \Div(d(x)\unknown \nabla\mu\Delta\mu)
        -d(x)(\Delta\mu)^2\unknown
        \\
        &\quad
        -\frac{1}{2}d(x)\Delta(|\nabla\mu|^2)\unknown
        +d(x)|\nabla^2\mu|^2\unknown. 
        \end{split}
    \end{equation}
    Therefore,  integrating on $\Omega$ of both side of \eqref{eq:nabla(d unknown^alpha)},  we have, 
    \begin{equation*}
        \begin{split}
            2\int_\Omega(\nabla (d(x)\unknown)\cdot\nabla\mu)\Delta\mu\, dx
            &=-2\int_\Omega d(x)(\Delta\mu)^2\unknown\, dx\\
            &\quad-\int_\Omega d(x)\Delta(|\nabla\mu|^2)\unknown\, dx\\
            &\quad+2\int_\Omega d(x)|\nabla^2\mu|^2\unknown\, dx, 
        \end{split}
    \end{equation*}
    since the integral of the first term in the right-hand side of \eqref{eq:nabla(d unknown^alpha)} vanishes by using the boundary condition of \eqref{eq:Nonlinear-Fokker-Planck} with the divergence theorem. Using \eqref{eq:2.Pohozaev-method-3}, we obtain \eqref{D(nabla f nablamu)Deltamu f^alpha-1}.
\end{proof}

We next calculate the second term on the right-hand side of \eqref{d^2F/dt^2_4}.

\begin{lem}
    \label{nabla^2mu's-term}
     Let $\unknown$ be a classical solution of \eqref{eq:Nonlinear-Fokker-Planck} on $\overline{\Omega}\times [0, \infty)$. Then
    \begin{equation}
    \label{nabla^2mu's-term_1}
        \begin{split}
            2\int_\Omega d(x)(\nabla \log\unknown\cdot \nabla^2\mu\nabla\mu)\unknown\, dx
            &
            =-2\int_\Omega (\nabla d(x)\cdot\nabla^2\mu\nabla\mu)\unknown\, dx
            \\
            &\quad
            +\int_{\partial\Omega} d(x)\unknown\nabla (|\nabla\mu|^2)\cdot \bnu\, d\sigma
            \\
            &\quad
            -\int_\Omega d(x) \Delta( |\nabla\mu|^2)\unknown\, dx.
        \end{split}
    \end{equation}
    
\end{lem}

\begin{proof}
    Taking the inner product of $\nabla^2\mu\nabla\mu$ both side of \eqref{eq:2.Pohozaev-method-3},  we have
    \begin{equation}
    \label{eq:3.pohozaev-method_2.0}
         d(x)(\nabla \unk \cdot \nabla^2\mu\nabla\mu)=(\nabla(d(x)\unknown)\cdot\nabla^2\mu\nabla\mu)-(\nabla d(x)\cdot \nabla^2\mu\nabla\mu)\unk.
    \end{equation}
    Next,  we compute $(\nabla(d(x)\unk)\cdot\nabla^2\mu\nabla\mu)$. Writing a vector in component form,  we obtain
    \begin{equation}
    \label{eq:3.Pohozaev-method_2.1}
    (\nabla(d(x)\unk)\cdot\nabla^2\mu\nabla\mu)
    =\sum_{i, j=1}^n(d(x)\unknown)_{x_i}\mu_{x_ix_j}\mu_{x_j}.
    \end{equation}
    We make a divergence form on the right-hand side of \eqref{eq:3.Pohozaev-method_2.1} as follows: 
    \begin{equation}
    \label{eq:3.pohozaev-method_2.2}
        (d(x)\unknown)_{x_i}\mu_{x_ix_j}\mu_{x_j}=(d(x)\unk\mu_{x_ix_j}\mu_{x_j})_{x_i}-d(x)\unk(\mu_{x_ix_j}\mu_{x_j})_{x_i}.
    \end{equation}
   Note that $\mu_{x_ix_j}\mu_{x_j}=\frac{1}{2}((\mu_{x_j})^2)_{x_i}$. Thus,  we arrive at 
   \begin{equation}
   \label{eq:3.pohozaev-method_2.3}
    \begin{split}
        (\nabla(d(x)\unk)\cdot\nabla^2\mu\nabla\mu)
        &
        =\sum_{i, j=1}^n\left(\frac{1}{2}(d(x)\unk(\mu_{x_j})^2_{x_i})_{x_i}-\frac{d(x)}{2}(\mu_{x_j})^2_{x_ix_i}\unknown\right)
        \\
        &
        =\frac{1}{2}\Div(d(x)\unknown\nabla|\nabla\mu|^2)-\frac{d(x)}{2}\Delta(|\nabla\mu|^2)\unknown
    \end{split}
   \end{equation}
   Therefore,  integrating on $\Omega$ of both side of \eqref{eq:3.pohozaev-method_2.3},  we have, 
   \begin{equation}
        \begin{split}
            2\int_\Omega (\nabla(d(x)\unknown)\cdot \nabla^2\mu\nabla\mu)\, dx
            &=
            \int_\Omega \Div(d(x)\unknown\nabla|\nabla\mu|^2)\, dx
            \\
            &\quad
            -\int_\Omega d(x) \Delta(|\nabla\mu|^2)\unk\, dx.
        \end{split}
   \end{equation}
   
    Using \eqref{eq:3.pohozaev-method_2.0} together with the divergence theorem,  we obtain \eqref{nabla^2mu's-term_1}. 
\end{proof}

Plugging \eqref{D(nabla f nablamu)Deltamu f^alpha-1} and \eqref{nabla^2mu's-term_1} into \eqref{d^2F/dt^2_4},  we obtain 
\begin{equation}
    \label{eq:exp_decay.rate_dissipation}
    \begin{split}
        \frac{d}{dt}\D[\unknown](t)
        &
        =-2\int_\Omega (\nabla\mu\cdot\nabla^2\phi(x)\nabla\mu)\unknown\, dx
        -
        2\int_\Omega d(x) |\nabla^2\mu|^2\unknown\, dx
        \\
        &\quad
        +\int_{\partial\Omega} d(x)\unknown\nabla (|\nabla\mu|^2)\cdot \bnu\, d\sigma
        \\
        &\quad
        +2\int_\Omega (\log\unknown-1)(\nabla d(x)\cdot\nabla^2\mu\nabla\mu)\unknown\, dx 
        \\
        &\quad
        +2\int_\Omega (\log\unknown+1)(\nabla d(x)\cdot\nabla\mu)\Delta\mu \unknown\, dx
        \\
        &\quad
        +2\int_\Omega \log\unknown(\nabla d(x)\cdot\nabla\mu)(\nabla \log\unknown\cdot\nabla\mu)\unknown\, dx
        \\
        &=:
        -2I_1-2I_2+I_3+2I_4+2I_5+2I_6.
    \end{split}
\end{equation}

Since $\unk$ is positive, we have $\nabla\mu\cdot\bnu=0$ on $\partial\Omega$. Then,  it is well-known that the outer normal derivative of $|\nabla\mu|^2$ can be written as
\begin{equation}
    \nabla|\nabla\mu|^2\cdot\bnu
    =
    2B_x(\nabla\mu, \nabla\mu), 
\end{equation}
at $x\in\partial\Omega$,  where $B_x$ is the second fundamental form at $x\in \partial \Omega$ (cf. \cite{MR555661}*{Lemma 5.3}, \cite{MR3348119}*{Lemma 4.2}). From the convexity assumption of $\Omega$,  the principal curvature of $\partial\Omega$ is non-positive thus we have
$I_3\leq0$. Therefore,  we obtain
\begin{equation}
    \label{second-derivative-F}
    \frac{d}{dt}\D[\unknown](t)\le -2I_1-2I_2+2I_4+2I_5+2I_6.
\end{equation}

\begin{remark}
    If $d(x)$ is constant,  then $\nabla d(x)=0$ so \eqref{second-derivative-F} can be written as
    \begin{equation}
        \frac{d}{dt}\D[\unknown](t)\le -2I_1-2I_2, 
    \end{equation}
    which was deduced by \cites{MR1842428,MR1639292,MR1812873}. The above computation is based on \cite{MR3497125}*{\S 2.1}. Inequality \eqref{second-derivative-F} is an extension of the previous result for the case where $d(x)$ is not constant.
    
    Using the assumption $\nabla^2 \phi(x) \ge \lambda I$, we have $\nabla\mu\cdot\nabla^2\phi(x)\nabla\mu\ge\lambda|\nabla\mu|^2$ hence together with $I_2 \ge 0$, we have
    \begin{equation}
        \frac{d}{dt}\D[\unknown](t)\le -\lambda \D[\unknown](t).
    \end{equation}
    Therefore, we obtain the exponential decay of $\D[\unknown](t)$ by from the Gronwall inequality. 
\end{remark}
Hereafter, we handle the terms $I_4, I_5$ and $I_6$. We prepare the following lemma. 

\begin{lem}
\label{lem:I5}
    Let $\unknown$ be a bounded, positive classical solution of \eqref{eq:Nonlinear-Fokker-Planck} on $\overline{\Omega}\times [0, \infty)$. Then, for any positive constant $\delta_1>0$ 
    \begin{equation}
    \label{eq:2.Est_Grad_Diffusion_1}
        \begin{split}
            &\quad
            \left|\int_\Omega (\log \unknown-1) (\nabla  d(x)\cdot \nabla^2\mu\nabla\mu)\unknown\, dx\right|
            \\
            &\le \frac{1}{2\delta_1}\left\|\sqrt{d}\nabla\log d\right\|_\infty^2(\|\log\unknown\|_\infty+1)^2\int_\Omega |\nabla\mu|^2\unknown\, dx
            \\
            &\quad
            +
            \frac{\delta_1}{2}\int_{\Omega} d(x)|\nabla^2\mu|^2\unknown\, dx.
        \end{split}
    \end{equation}
\end{lem}

\begin{proof}
    From the triangle inequality, we have
    \begin{equation*}
        \begin{split}
            &\quad 
            \left|\int_\Omega (\log \unknown-1)(\nabla d(x)\cdot\nabla^2\mu\nabla\mu)\unknown\, dx\right|
            \\
            &\le
            \int_\Omega  (|\log\unknown|+1)|\nabla d(x)||\nabla^2\mu||\nabla\mu|\unknown\, dx.
        \end{split}
    \end{equation*}
    Since $d(x)>0$,  it follows by H\"{o}lder's inequality and Young's inequality that for any positive constant $\delta_1>0$ 
    \begin{equation*}
        \begin{split}
        &\quad
           \int_\Omega (|\log\unknown|+1)|\nabla d(x)||\nabla^2\mu||\nabla\mu|\unknown \, dx
           \\
           &\le 
           \left(\int_\Omega \frac{1}{d(x)}|\nabla d(x)|^2(|\log \unknown|+1)^2|\nabla\mu|^2\unknown\, dx\right)^{\frac{1}{2}}\left(\int_\Omega
           d(x)|\nabla^2\mu|^2\unknown\, dx\right)^{\frac{1}{2}}
           \\
           &\le
           \frac{1}{2\delta_1}\left\|\frac{\nabla d}{\sqrt{d}}\right\|_\infty^2(\|\log\unknown\|_\infty+1)^2\int_\Omega|\nabla\mu|^2\unknown\, dx
            \\
            &\quad
            +\frac{\delta_1}{2}\int_\Omega d(x) |\nabla^2\mu|^2\unknown\, dx.
        \end{split}
    \end{equation*}
    Therefore using $d(x)\nabla \log d(x)=\nabla d(x)$, we obtain \eqref{eq:2.Est_Grad_Diffusion_1} .
\end{proof}

From \eqref{eq:2.Est_Grad_Diffusion_1},  we obtain

\begin{equation}
    \label{eq:2.Estimate_I5}
    2|I_4|
    \le 
    \frac{1}{\delta_1}\left\|\sqrt{d}\nabla \log d\right\|_\infty^2(\|\log\unknown\|_\infty+1)^2
    \D[\unk](t)
    +
    \delta_1I_2.
\end{equation}

Next, we estimate $I_5$ by $\D[\unk](t)$ and $I_2$.

\begin{lem}
    \label{lem:I6}
    Let $\unknown$ be a bounded, positive classical solution of \eqref{eq:Nonlinear-Fokker-Planck} on $\overline{\Omega}\times [0, \infty)$. Then for any positive constant $\delta_2>0$,
    \begin{equation}
        \label{eq:2.Est_Grad_Diffusion_2}
        \begin{split}
        &\quad\left|\int_\Omega (\log\unknown+1)(\nabla  d(x)\cdot\nabla\mu)\Delta\mu \unknown\, dx\right|\\
        &\le
        \frac{1}{2\delta_2}(\|\log\unknown\|_\infty+1)^2\|\sqrt{d}\nabla \log d\|_\infty^2\int_\Omega |\nabla\mu|^2\unknown\, dx
        \\
        &\quad
        +\frac{\delta_2 n}{2}\int_\Omega d(x)|\nabla^2\mu|^2\unknown \, dx.
        \end{split}
    \end{equation}
\end{lem}

\begin{proof}
    From the triangle inequality, we have
    \begin{equation*}
        \begin{split}
            &\quad\left|\int_\Omega  (\log\unknown+1)(\nabla d(x)\cdot\nabla\mu)\Delta\mu \unknown\, dx\right|
            \\
            &\le
            \int_\Omega (|\log \unknown|+1) |\nabla  d(x)||\nabla\mu||\Delta\mu|\unknown\, dx .
        \end{split}
    \end{equation*}
    Similarly, in the proof of Lemma \ref{lem:I5},  it follows from H\"older's and Young's inequality that
    \begin{equation*}
        \begin{split}
            &\quad
            \int_\Omega (|\log \unknown|+1) |\nabla d(x)||\nabla\mu||\Delta\mu|\unknown\, dx 
            \\
            &\le 
            \left(\int_\Omega \frac{1}{d(x)}(|\log\unknown|+1)^2|\nabla d(x)|^2|\nabla\mu|^2\unknown\, dx\right)^{\frac{1}{2}}
            \\
            &\quad\times\left(\int_\Omega d(x)(\Delta\mu)^2\unknown \, dx\right)^{\frac{1}{2}} 
            \\
            &\le 
            \frac{1}{2\delta_2}(\|\log \unknown\|_\infty+1)^2\left\|\sqrt{d}\nabla \log d\right\|_\infty^2\int_\Omega |\nabla\mu|^2\unknown\, dx
            \\
            &\quad
            +\frac{\delta_2}{2}\int_\Omega d(x)(\Delta\mu)^2\unknown \, dx.
        \end{split}
    \end{equation*}
    Since $(\Delta\mu)^2\le n|\nabla^2\mu|^2$,
    \begin{equation*}
        \int_\Omega d(x)(\Delta\mu)^2\unknown \, dx
            \le n\int_\Omega d(x)|\nabla^2\mu|^2\unknown \, dx.
    \end{equation*}
    Therefore,  \eqref{eq:2.Est_Grad_Diffusion_2} follows from summarizing the above estimates.
\end{proof}

From \eqref{eq:2.Est_Grad_Diffusion_2},  we obtain
\begin{equation}
    \label{eq:2.Estimate_I6}
    2|I_5|
    \leq \frac{1}{\delta_2}   
    (\|\log\unknown\|_\infty+1)^2\left\|\sqrt{d}\nabla\log d\right\|_\infty^2\D[\unk](t)
    +
    n\delta_2I_2.
\end{equation}
To proceed with the estimate $I_6$,  we first substitute $\nabla\unk$ by $\nabla\mu$. In the next lemma,  we use the relation \eqref{mu}.

\begin{lem}
    \label{lem:I7}
    Let $\unknown$ be a classical solution of \eqref{eq:Nonlinear-Fokker-Planck} on $\overline{\Omega}\times [0, \infty)$. Then
    \begin{equation}
    \label{eq:2.Compute_Grad_Diffusion_3}
        \begin{split}
            &\quad
           \int_\Omega \log\unknown(\nabla d(x)\cdot\nabla\mu)(\nabla\log \unknown\cdot\nabla\mu)\unknown\,dx
            \\
            &=
            \int_\Omega\log \unknown(\nabla \log d(x)\cdot \nabla\mu)|\nabla\mu|^2\unknown\, dx
            \\
            &\quad
            -\int_\Omega d(x)(\nabla \log d(x)\cdot \nabla\mu)^2(\log\unknown)^2\unknown\, dx
            \\
            &\quad
            -\int_\Omega \log \unknown(\nabla\log  d(x)\cdot\nabla\mu)(\nabla\phi(x)\cdot\nabla\mu)\unknown\, dx.
        \end{split}
    \end{equation}
\end{lem}

\begin{proof}
    Taking the gradient of both side of \eqref{mu}, we have
    \begin{equation}
        \nabla\mu
        =
         d(x)\nabla\log\unknown+ \log\unknown\nabla d(x)+\nabla\phi(x).
    \end{equation}
    Thus,  the integrand of $I_6$ turns into
    \begin{equation}
        \begin{split}
        &\quad
       \log \unknown(\nabla  d(x)\cdot \nabla\mu)(\nabla \log\unknown\cdot \nabla\mu)\unknown
        \\
        &=\log \unknown(\nabla \log d(x)\cdot \nabla\mu)(d(x)\nabla \log\unknown\cdot \nabla\mu)\unknown
        \\
        &=
        \log\unknown
        (\nabla\log d(x)\cdot \nabla\mu)
        \left(
        \left(
        \nabla\mu
        -
        d(x)\log\unknown\nabla \log d(x)
        -
        \nabla\phi(x)
        \right)
        \cdot \nabla\mu\right)
        \unknown
        \end{split}
     \end{equation}
    Taking the integration on $\Omega$ on both sides, we obtain \eqref{eq:2.Compute_Grad_Diffusion_3}.
\end{proof}

Note that the second term of the right-hand side of \eqref{eq:2.Compute_Grad_Diffusion_3} is non-positive.  we have from \eqref{eq:2.Compute_Grad_Diffusion_3} that
\begin{equation}
    \label{eq:2.estimate_I7}
    \begin{split}
        2 I_6
        &\leq
        2\int_\Omega \log\unknown (\nabla \log d(x)\cdot \nabla\mu)|\nabla\mu|^2\unknown\, dx
        \\
        &\quad
        -2\int_\Omega \log\unknown(\nabla\log d(x)\cdot\nabla\mu)(\nabla\phi(x)\cdot\nabla\mu)\unknown\, dx
        \\
        &\le
        2\|\nabla \log d\|_\infty\|\log \unknown\|_\infty\int_\Omega |\nabla\mu|^3\unknown\, dx
        \\
        &\quad
        +2\|\nabla\log d\|_\infty \|\nabla\phi\|_\infty\|\log\unknown\|_\infty\D[\unk](t).    
    \end{split}
\end{equation}

We need to handle a cubic nonlinearity in the right-hand side of \eqref{eq:2.estimate_I7}. 
Compared to the problem in \cites{MR4506846,arXiv:epshteyn2024longtimeasymptoticbehaviornonlinear}, we subject the Neumann boundary condition to $\mu$. Thus, we cannot use the Sobolev inequality of homogeneous type
\begin{equation*}
    \int_\Omega |\bu|^{p^*}\,dx \le C\int_\Omega |\nabla\bu|^2\,dx,\quad \frac{1}{p^*}=\frac{1}{2}-\frac{1}{n},
\end{equation*}
Instead, we use the following Sobolev-Poincar\'e type inequality. 

\begin{prop}
\label{Sobolev}
    Let $\unknown$ be a strictly positive bounded classical solution of 
\eqref{eq:Nonlinear-Fokker-Planck} in $\overline{\Omega}\times [0, \infty)$. Then,  there is a suitable positive constant $\Cl{const:sobolevtype}>0$ depending only on $n$, $\Cr{const:f>c}$, $\Cr{const:f<c}$ and $\Omega$, such that for any vector field $\bv\in C^1(\Omega)$, 
    \begin{equation}
    \label{Sobolev-Poincare}
        \left(\int_\Omega |\bv-\overline{\bv}|^{p^*}\unknown\, dx\right)^{\frac{1}{p^*}}\le \Cr{const:sobolevtype}\left(\int_\Omega |\nabla\bv|^2\unknown\, dx\right)^{\frac{1}{2}}, 
    \end{equation}
    where $\overline{\bv}$ is the integral average of $\bv$ and the $p^*$ is an exponent satisfying $\frac{1}{p^*}=\frac{1}{2}-\frac{1}{n}$ for $n\ge3$ and arbitrary $2 \le p^* < \infty$ for $n=1, 2$.
\end{prop}

\begin{proof}
    Since $p^*$ is the Sobolev exponent,  it follows from the Sobolev-Poincar\'e inequality (\cite{MR1814364}*{p.174}, \cite{MR1817225}*{Theorem 4.3}) that
    \begin{equation*}
        \left(\int_\Omega |\bv-\overline{\bv}|^{p^*}\, dx\right)^{\frac{1}{p^*}}\le \Cl{const:Sobolev}\left(\int_\Omega |\nabla\bv|^2\, dx\right)^{\frac{1}{2}}
    \end{equation*}
    holds for any vector field $\bv \in C^1(\Omega)$. By the definition of $\Cr{const:f>c}$,  $\Cr{const:f<c}$,  we obtain that
    \begin{equation*}
        \begin{split}
            \left(\int_\Omega |\bv-\overline{\bv}|^{p^*}\unknown\, dx\right)^{\frac{1}{p^*}}
            &\le
            \Cr{const:f<c}^{\frac{1}{p^*}}\left(\int_\Omega |\bv-\overline{\bv}|^{p^*}\, dx\right)^{\frac{1}{p^*}}\\
            &\le 
            \Cr{const:f<c}^{\frac{1}{p^*}}\Cr{const:Sobolev}\left(\int_\Omega |\nabla\bv|^2\, dx\right)^{\frac{1}{2}}
            \\
            &\le 
            \frac{\Cr{const:f<c}^{\frac{1}{p^*}}\Cr{const:Sobolev}}{\Cr{const:f>c}^{\frac{1}{2}}}\left(\int_\Omega |\nabla\bv|^2\unknown\, dx\right)^{\frac{1}{2}}.
        \end{split}
    \end{equation*}
\end{proof}

Next, we prove an interpolation inequality from the Sobolev-Poincar\'e type inequality \eqref{Sobolev-Poincare}.

\begin{prop}
\label{prop:|v|^3}
    Let $n=1, 2, 3$. Let $\unknown$ be a bounded, strictly positive solution of \eqref{eq:Nonlinear-Fokker-Planck} on $\overline{\Omega}\times [0, \infty)$. Then,  there are constants $\Cl{const:|v|^3->|nablav|^2}$,  $\Cl{const:|v|^3-|v|^2, 1}$,  and $\Cl{const:|v|^3-|v|^2, 2}>0$ such that for any $\bv\in C^1(\Omega)$, 
    \begin{equation}
        \label{eq:2.Interpolation}
            \int_\Omega |\bv|^3\unknown\, dx
	    \le
	    \Cr{const:|v|^3->|nablav|^2}
	    \int_\Omega |\nabla\bv|^2\unknown\, dx
	    +
	    \Cr{const:|v|^3-|v|^2, 1}
	    \left(\int_\Omega |\bv|^2\unknown\, dx\right)^3
            +
	    \Cr{const:|v|^3-|v|^2, 2}
	    \left(\int_\Omega |\bv|^2\unknown\, dx\right)^{\frac{3}{2}}.
    \end{equation}
\end{prop}

\begin{proof}
    This assertion was proved in \cite{araki2025longtimebehaviorfreeenergy}.
    Here we demonstrate it for the sake of completeness.
    Let $a, b>0$ such that $a+b=1$,  and let $p>1$ satisfying $3ap\geq1$. Note that $|\bv|^{3}=|\bv|^{3a}|\bv|^{3b}$, hence by H\"{o}lder's and the convex inequality, 
    \begin{equation}
    \label{|v|^3f'sholder}
        \begin{split}
            \int_\Omega |\bv|^3\unknown\, dx&\le \left(\int_\Omega |\bv|^{3ap}\unknown\, dx \right)^{\frac{1}{p}}\left(\int_\Omega |\bv|^{3bp'}\unknown\, dx\right)^{\frac{1}{p'}} \\
            &=\left(\int_\Omega |\bv-\overline{\bv}+\overline{\bv}|^{3ap}\unknown\, dx \right)^{\frac{1}{p}}\left(\int_\Omega |\bv|^{3bp'}\unknown\, dx\right)^{\frac{1}{p'}} \\
            &\le 2^{\frac{3ap-1}{p}}\Biggl(\left(\int_\Omega |\bv-\overline{\bv}|^{3ap}\unknown\, dx \right)^{\frac{1}{p}}
            \\
            &\quad
            +\left(\int_\Omega |\overline{\bv}|^{3ap}\unknown\, dx \right)^{\frac{1}{p}}\Biggr)\left(\int_\Omega |\bv|^{3bp'}\unknown\, dx\right)^{\frac{1}{p'}}, 
        \end{split}
    \end{equation}
    where $p'$ is the H\"{o}lder dual index of $p$. Next, we set $3ap=p^*\geq1$. By Proposition \ref{Sobolev}, 
    \begin{equation*}
        \left(\int_\Omega |\bv-\overline{\bv}|^{3ap}\unknown\, dx \right)^{\frac{1}{p}}
        \le 
        \Cr{const:sobolevtype}^{\frac{p^*}{p}}\left(\int_\Omega |\nabla\bv|^{2}\unknown\, dx \right)^{\frac{p^*}{2p}}.
    \end{equation*}
    We take $3bp'=2$ and $\frac{p^*}{2p}<1$. Then,  by using Young's inequality, 
    \begin{equation*}
    \begin{split}
        \left(\int_\Omega |\nabla\bv|^{2}\unknown\, dx \right)^{\frac{p^*}{2p}}\left(\int_\Omega |\bv|^{2}\unknown\, dx\right)^{\frac{1}{p'}}
            &
            \le
            \frac{p^*}{2p}\int_\Omega |\nabla\bv|^2\unknown\, dx
            \\
            &\quad
            +\left(1-\frac{p^*}{2p}\right)\left(\int_\Omega |\bv|^2\unknown\, dx\right)^{\frac{1}{p'}\left(1-\frac{p^*}{2p}\right)^{-1}}.
    \end{split}
    \end{equation*}
    Note from \eqref{eq:apriori} that $\Cr{const:f>c}$ is the minimum of $\unknown$ on $\overline{\Omega} \times [0, \infty)$. Then,  from H\"{o}lder's inequality and \eqref{eq:f_pdf} that
    \begin{equation*}
        \begin{split}
            \left(\int_\Omega |\overline{\bv}|^{p^*}\unknown\, dx \right)^{\frac{1}{p}}
            =|\overline{\bv}|^{\frac{p^*}{p}}
            &\le
            \left(\frac{1}{|\Omega|}\int_\Omega |\bv|\, dx\right)^{\frac{p^*}{p}}
            \\
            &\le
            \left(\frac{1}{|\Omega|\Cr{const:f>c}}\int_\Omega |\bv|\unknown\, dx\right)^{\frac{p^*}{p}}
            \\
            &\le 
            \left(\frac{1}{|\Omega|\Cr{const:f>c}}\right)^{\frac{p^*}{p}}\left(\int_\Omega |\bv|^2\unknown\, dx\right)^{\frac{p^*}{2p}}.
        \end{split}
    \end{equation*}
    Therefore subsituting the above inequality to \eqref{|v|^3f'sholder},  we obtain
    \begin{equation*}
        \begin{split}
            \int_\Omega |\bv|^3\unknown\, dx
            &
            \le 2^{\frac{3ap-1}{p}}\Cr{const:sobolevtype}^{\frac{p^*}{p}}\frac{p^*}{2p}\int_\Omega |\nabla\bv|^2\unknown\, dx
            \\
            &\quad
            +2^{\frac{3ap-1}{p}}
            \Cr{const:sobolevtype}^{\frac{p^*}{p}}\left(1-\frac{p^*}{2p}\right)\left(\int_\Omega |\bv|^2\unknown\, dx\right)^{\frac{1}{p'}\left(1-\frac{p^*}{2p}\right)^{-1}} 
            \\
            &\quad
            +2^{\frac{3ap-1}{p}}\left(\frac{1}{|\Omega|\Cr{const:f>c}}\right)^{\frac{p^*}{p}}\left(\int_\Omega |\bv|^2\unknown\, dx\right)^{\frac{p^*}{2p}+\frac{1}{p'}}.
        \end{split}
    \end{equation*}
    Next, we check the constraints' condition. If $n\ge3$,  note that $a+b=1$ and $p'$ is the H\"{o}lder dual index of $p$,  $3ap=p^*$ and $p^*$ is the Sobolev exponent. Then,  we get
    \begin{equation*}
            1=\frac{1}{p}+\frac{1}{p'}=\frac{3a}{p^*}+\frac{3b}{2}=\frac{3}{2}-\frac{3a}{n}.
    \end{equation*}
    Thus,  we obtain $a=\frac{n}{6}$. Combining $\frac{p^*}{2p}<1$ and $3ap=p^*$,  we deduce $a<\frac{2}{3}$ hence $n<4$,  which means $n=3$. If $n=1, 2$,  we put $p^*=6$. Then we deduce from $3ap=6$,  $3bp'=2$ that
    \begin{equation}
        1=\frac{1}{p}+\frac{1}{p'}=\frac{a}{2}+\frac{3}{2}b=\frac{1}{2}+b, 
    \end{equation}
    thus $a=b=\frac{1}{2}$,  $p=4$,  $p'=\frac{4}{3}$, $p^*=6$ and $\frac{1}{p'}(1-\frac{p^*}{2p})^{-1}=3$.
    Using the above results,  we obtain \eqref{eq:2.Interpolation},  where
    \begin{equation*}
            \Cr{const:|v|^3->|nablav|^2}:=2^{-\frac{3}{4}}3\Cr{const:sobolevtype}^{\frac{3}{2}}, \quad
            \Cr{const:|v|^3-|v|^2, 1}:=2^{-\frac{3}{4}}\Cr{const:sobolevtype}^{\frac{3}{2}},  \quad
            \Cr{const:|v|^3-|v|^2, 2}:=2^{\frac{5}{4}}\left(\frac{1}{|\Omega|\Cr{const:f>c}}\right)^{\frac{3}{2}}.
    \end{equation*}
\end{proof}

\begin{remark}
    In the proof of Proposition \ref{Sobolev} and \ref{prop:|v|^3}, we essentially used \eqref{eq:apriori}, the positive lower bounds and the upper bounds of the solution $\unknown$. Thus, the assertions of Proposition  \ref{Sobolev} and \ref{prop:|v|^3} can be extended to arbitrary integrable function $\unknown\in L^1(\Omega)$ satisfying \eqref{eq:apriori}.
\end{remark}

Using Lemma \ref{lem:I5},  \ref{lem:I6},  \ref{lem:I7} and Proposition \ref{prop:|v|^3} to \eqref{second-derivative-F},  we obtain the following estimate:

\begin{lem}
 \label{D>0} Let $n=1, 2, 3$.  Let $\unknown$ be a bounded,  positive
 classical solution of \eqref{eq:Nonlinear-Fokker-Planck} on
 $\overline{\Omega}\times [0, \infty)$. Then,  there is a small enough number $\Cr{const:nabla_log_d<c}>0$ such that if $\|\nabla \log d\|_\infty <\Cr{const:nabla_log_d<c}$, then there exist positive constants $\Cl{const:|nablamu|^2f^3/2}$, 
 $\Cl{const:|nablamu|^2f^3}>0$
 such that for $t>0$,
 \begin{equation}
  \label{eq:2.Second_time_derivative_F_Final}
  \frac{d}{dt}\D[\unknown](t)
    \le
    -\lambda
    \D[\unk](t)
    +
    \Cr{const:|nablamu|^2f^3/2}
    \left(\D[\unk](t)\right)^3
    +
    \Cr{const:|nablamu|^2f^3}
    \left(\D[\unk](t)\right)^{\frac{3}{2}}.
 \end{equation}
\end{lem}

\begin{proof}
Plugging \eqref{eq:2.Estimate_I5},  \eqref{eq:2.Estimate_I6},  and \eqref{eq:2.estimate_I7} into \eqref{second-derivative-F}, we can estimate the time derivative of $\D[\unknown](t)$ as
\begin{equation}
\label{Second-Order-Derivative_of_F}
    \begin{split}
        \frac{d}{dt}\D[\unknown](t)
        &\le
        -2I_1
        -(2-\delta_1-n\delta_2)I_2
        \\
        &\quad+
        \biggl(
        \left(\frac{1}{\delta_1}+\frac{1}{\delta_2}\right)(\|\log \unknown\|_\infty+1)^2\left\|\sqrt{d}\nabla\log d\right\|_\infty^2
        +
        \\
        &\qquad\qquad
        2\|\nabla\phi\|_\infty\|\log\unknown\|_\infty\|\nabla\log d\|_\infty
        \biggr)
        \D[\unk](t)
        \\
        &\quad+
        2\|\nabla \log d\|_\infty\|\log\unknown\|_\infty\int_\Omega |\nabla\mu|^3\unknown\, dx, 
    \end{split}
\end{equation}

From proposition \ref{prop:|v|^3} with $\bv=\nabla\mu$ and using \eqref{assume:d>c},  we have
\begin{equation}
\label{|nablamu|^3}
            \int_\Omega |\nabla \mu|^3\unknown\, dx
            \le 
            \frac{\Cr{const:|v|^3->|nablav|^2}}{\Cr{const:d>c}}I_2
            +
            \Cr{const:|v|^3-|v|^2, 1}\left(\D[\unk](t)\right)^3
            +\Cr{const:|v|^3-|v|^2, 2}\left(\D[\unk](t)\right)^{\frac{3}{2}}.
\end{equation}
Plugging \eqref{|nablamu|^3} into \eqref{Second-Order-Derivative_of_F},  we obtain
\begin{equation} 
    \label{eq:2.Second-Order-Derivative_of_F_2}
    \begin{split}
    \frac{d}{dt}\D[\unknown](t)
    &
    \le
    -2I_1
    -
    \left(2-\delta_1-n\delta_2-\frac{\Cr{const:f, nablaD, nablaphi1}}{\Cr{const:d>c}}\|\nabla\log d\|_\infty\right)I_2
    \\
    &\quad
    +\biggl(
    \left(\frac{1}{\delta_1}+\frac{1}{\delta_2}\right)(\|\log \unknown\|_\infty+1)^2\left\|\sqrt{d}\nabla\log d\right\|_\infty^2
    \\
    &\qquad\qquad
    +2\|\nabla\phi\|_\infty\|\log\unknown\|_\infty\|\nabla\log d\|_\infty
    \biggr)
        \D[\unk](t)
    \\
    &\quad
    +\Cr{const:f, nablaD, nablaphi2}\|\nabla\log d\|_\infty
    \left(\D[\unk](t)\right)^3
    \\
    &\quad
    +\Cr{const:f, nablaD, nablaphi3}\|\nabla\log d\|_\infty
    \left(\D[\unk](t)\right)^{\frac{3}{2}}, 
    \end{split}
    \end{equation}
where
\begin{equation}
    \begin{split}
        \Cl{const:f, nablaD, nablaphi1}&:=2\|\log\unknown\|_\infty\Cr{const:|v|^3->|nablav|^2},
        \\
        \Cl{const:f, nablaD, nablaphi2}
        &:=2\|\log\unknown\|_\infty\Cr{const:|v|^3-|v|^2, 1}, 
        \\
        \Cl{const:f, nablaD, nablaphi3}
        &:=2\|\log\unknown\|_\infty\Cr{const:|v|^3-|v|^2, 2}.
    \end{split}
\end{equation}
Since $\nabla^2 \phi(x)\ge \lambda I$ for $x\in\Omega$, $I_1$ can be estimated by
\begin{equation}
    \label{eq:2.Estimate.Hesse.potential}
    2I_1
    \ge
    2\lambda \int_\Omega |\nabla\mu|^2\unknown\, dx
    =
    2\lambda \D[\unk](t).
\end{equation}
Plugging \eqref{eq:2.Estimate.Hesse.potential} into
\eqref{eq:2.Second-Order-Derivative_of_F_2},  we obtain 
\begin{equation*}
    \begin{split}
        \frac{d}{dt}\D[\unknown](t)
        &\le
        -\biggl(2\lambda-
        \biggl(\left(\frac{1}{\delta_1}+\frac{1}{\delta_2}\right)(\|\log \unknown\|_\infty+1)^2\left\|\sqrt{d}\nabla\log d\right\|_\infty^2
        \\
        &\qquad\qquad
        +2\|\nabla\phi\|_\infty\|\log\unknown\|_\infty\|\nabla\log d\|_\infty
        \biggr)
        \biggr)\D[\unk](t)
        \\
        &\quad
        -(2-\delta_1-n\delta_2-\frac{\Cr{const:f, nablaD, nablaphi1}}{\Cr{const:d>c}}\|\nabla\log d\|_\infty)I_2
        \\
        &\quad+\Cr{const:f, nablaD, nablaphi2}\|\nabla\log d\|_\infty
    \left(\D[\unk](t)\right)^3
    \\
    &\quad
     +\Cr{const:f, nablaD, nablaphi3}\|\nabla\log d\|_\infty
    \left(\D[\unk](t)\right)^{\frac{3}{2}}.
    \end{split}
\end{equation*}
Let $\delta_1$, $\delta_2$ be small enough such that $\delta_1+n\delta_2=1$. For example, let $\delta_1=1/2$ and $\delta_2=1/2n$.
 By taking $\|\nabla \log d\|_\infty<\Cr{const:nabla_log_d<c}$ sufficiently small, we obtain 
 \begin{equation*}
    \begin{split}
        -\biggl(2\lambda-
        \biggl(
        \left(\frac{1}{\delta_1}+\frac{1}{\delta_2}\right)(\|\log \unknown\|_\infty+1)^2\left\|\sqrt{d}\nabla\log d\right\|_\infty^2
        \\
        +
        2\|\nabla\phi\|_\infty\|\log\unknown\|_\infty\|\nabla\log d\|_\infty
        \biggr)
        \biggr)
        &\le
        -\lambda, 
        \\
        1-\frac{\Cr{const:f, nablaD, nablaphi1}}{\Cr{const:d>c}}\|\nabla\log d\|_\infty
        &\ge
        0.
    \end{split}
 \end{equation*}
 Here, the constant $\Cr{const:nabla_log_d<c}$ depends on $\unknown_0$, $\nabla\phi$, $n$, $\Omega$, and $\Cr{const:d>c}$.
 Then, the time derivative of $\D[\unk](t)$ can be estimated as
 \begin{equation*}
    \begin{split}
        \frac{d}{dt}\D[\unknown](t)
        &
        \le
        -\lambda
        \D[\unk](t)
        +
        \Cr{const:f, nablaD, nablaphi2}\|\nabla\log d\|_\infty
        \left(\D[\unk](t)\right)^3
        \\
        &\quad
        +
        \Cr{const:f, nablaD, nablaphi3}\|\nabla\log d\|_\infty
        \left(\D[\unk](t)\right)^{\frac{3}{2}}.
    \end{split}
 \end{equation*}

 Thus,  we obtain \eqref{eq:2.Second_time_derivative_F_Final} by taking
 constants as 
 \begin{equation*}
  \Cr{const:|nablamu|^2f^3/2}
   :=
   \Cr{const:f, nablaD, nablaphi2}\|\nabla\log d\|_\infty, 
   \quad
   \Cr{const:|nablamu|^2f^3}
   :=
   \Cr{const:f, nablaD, nablaphi3}\|\nabla\log d\|_\infty.
 \end{equation*}
\end{proof}

From differential inequality
\eqref{eq:2.Second_time_derivative_F_Final},  we use the following Gronwall type lemma. See~\cite{araki2025longtimebehaviorfreeenergy} for the complete proof.

\begin{lemma}
 \label{lem:2.Gronwall_Lemma}
 Let $g:[0, \infty)\rightarrow\R$ be a differentialble function. Assume
 there exist positive constants $\Cl{const:Gronwall1}, 
 \Cl{const:Gronwall2}$,  and $\Cl{const:Gronwall3}>0$ such that
 \begin{equation}
  \label{ineq:gronwall_of_g}
   \frac{d}{dt}g(t)
   \le
   -
   \Cr{const:Gronwall1} g(t)
   + 
   \Cr{const:Gronwall2}
   g(t)^{\frac{3}{2}}
   +
   \Cr{const:Gronwall3}
   g(t)^3
 \end{equation}
 for any $t>0$. Then,  there exist positive constants
 $\Cl{const:Gronwall_Initial},  \Cl{const:Gronwall_Coefficient}>0$
 depending only on $\Cr{const:Gronwall1},  \Cr{const:Gronwall2}$,  and
 $\Cr{const:Gronwall3}>0$ such that if
 $g(0)<\Cr{const:Gronwall_Initial}$,  then $g(t)\leq
 \Cr{const:Gronwall_Coefficient}e^{-\Cr{const:Gronwall1}t}$
 holds for any $t>0$.
\end{lemma}

Now,  we are in a position to demonstrate the main theorem.

\begin{proof}[Proof of Theorem \ref{thm:exp_decay}]
 
 Define $g(t)$ by
 \begin{equation}
  g(t):=
   \D[\unk](t)
   =
   \int_\Omega
   |\nabla\mu|^2\unk\, dx.
 \end{equation}
 From \eqref{eq:2.Second_time_derivative_F_Final}, we have
 \begin{equation}
  g(t)
   \leq
   -
   \lambda
   g(t)
   +
   \Cr{const:|nablamu|^2f^3/2}
   \left(g(t)\right)^3
   +
   \Cr{const:|nablamu|^2f^3}
   \left(g(t)\right)^{\frac{3}{2}}.
 \end{equation} 
 Then,  we obtain Theorem \ref{thm:exp_decay} by applying the Gronwall lemma
 (Lemma \ref{lem:2.Gronwall_Lemma}) with $\Cr{const:Gronwall1}=\lambda, 
 \Cr{const:Gronwall2}=\Cr{const:|nablamu|^2f^3}$,  and
 $\Cr{const:Gronwall3}=\Cr{const:|nablamu|^2f^3/2}$.
\end{proof}

\section{Numerical Analysis}
\label{sec:Numerical_Analysis}

In this section, we numerically compute $1$-dimensional \eqref{eq:Nonlinear-Fokker-Planck} in case $\Omega=(0,1)$ and observe the relationship between the spatial inhomogeneity of the diffusion and the decay rate of the energy dissipation. To compute \eqref{eq:Nonlinear-Fokker-Planck} numerically, we first consider the change of variable $\mu=d(x)\log \unknown+\phi(x)$, the potential function of the velocity. Then the potential function $\mu$ satisfies 
\begin{equation}
\label{eq:mu_s_equation}
    \mu_t
    =
    d(x)\mu_{xx}
    -
    (
    (\mu-\phi(x))(\log d(x))_x
    +
    \phi_x(x)
    )
    \mu_x
    +
    |\mu_x|^2, 
\end{equation}
for $0<x<1,\ t>0$, subjected to the boundary condition $\unk(x,t)\mu_x(x,t)=0$ at $x=0,1$ and $t>0$.
Note that \eqref{eq:mu_s_equation} is a homogeneous Neumann boundary value problem, hence we can use the standard forward Euler scheme: For $x\in (0,1)$
and small parameters $dt$, $dx>0$, 
\begin{equation}
    \label{eq:4.Forward_Euler_Scheme}
    \mu(x,t+dt)
    =
    \mu(x,t)
    +
    dt
    \left(
    d(x)\delta_x^2 \mu(x,t)
    +
    F(\mu(x,t),\delta_x\mu(x,t))
    \right),
\end{equation}
where $\delta_x, \delta_x^2$ are the first and second "central differences" subject to the Neumann boundary condition, that is,
\begin{equation*}
\begin{split}
    \delta_x f(x)&=\frac{f(x+dx)-f(x-dx)}{2dx},
    \\
    \delta_x^2 f(x)&=\frac{f(x+dx)+f(x-dx)-2f(x)}{(dx)^2},
\end{split}
\end{equation*}
and nonlinear term $F$ is given by
\begin{multline}
    F(\mu(x,t),\delta_x\mu(x,t))
    \\
    =
    -
    (
    (\mu(x,t)-\phi(x))(\log d(x))_x
    +
    \phi_x(x)
    )
    \delta_x\mu(x,t)
    +
    |\delta_x\mu(x,t)|^2.
\end{multline}
Since $\mu_x(0,t)=\mu_x(1,t)=0$ for $t>0$, we introduce virtual points $x=-dx$ and $x=1+dx$ such that
\begin{equation}
    \mu(dx,t)-\mu(-dx,t)=0,\qquad
    \mu(1+dx,t)-\mu(1-dx,t)=0,\qquad
\end{equation}
and the scheme \eqref{eq:4.Forward_Euler_Scheme} enforces the boundary $x=0,1$.

We define $\unknown$ from $\mu$ as
\begin{equation}
\label{eq:mu_to_unknown}
    \unknown(x,t)
    =
    \exp
    \left(
        \frac{\mu(x,t)-\phi(x)}{d(x)}
    \right)
\end{equation}
and compute the energy dissipation
\begin{equation}
\label{eq:1-dimesional_dissipation_func}
    \D[\unknown]=\int_0^1|\mu_x(x,t)|^2\unknown(x,t)\,dx
\end{equation}
by the trapezoid integral. We will seek the effect of the spatial inhomogeneity, so we use the spatial inhomogeneity of the diffusion as 
\begin{equation}
\label{eq:diffusion_coeff}
    d(x)=\dcenter+\dcoeff\sin(\dosc x)
\end{equation}
for some positive constants $\dcenter$, $\dcoeff$, and $\dosc$. The space mesh $dx$, time step size $dt$, and the potential $\phi$ are fixed in this simulation, given by
\begin{equation}
    dx
    =
    \frac{1}{100},
    \qquad
    dt
    =
    \frac{dx^2}{6}
    =
    \frac{1}{60000}
    ,
    \qquad
    \phi(x)
    =
    \frac{1}{2}x^2.
\end{equation}
The initial data $\mu_0(x)$ is given by
\begin{equation}
    \label{eq:4.InitialData}
    \mu_0(x)
    =
    1.0
    +
    0.5
    \cos(\mu_{\text{osc}}x),
\end{equation}
where $\mu_{\text{osc}}$ is a positive constant.

The decay rate of $\D[\unk](t)$ is calculated as follows. While $\D[\unk](t)> 10^{-10}$, compute the decay rate of $\D[\unk](t)$ as
\begin{equation}
    \label{eq:5.Def_DecayRate}
    a=
    \min_{t}
    \frac{\log_{10}\D[\unk](t)-\log_{10}\D[\unk](t+dt)}{dt},
\end{equation}
which means $\D[\unk](t)=\Cr{const:4.Coeff}10^{-at}$ for some positive constant $\Cl{const:4.Coeff}>0$. Hereafter, we call the number $a$ defined \eqref{eq:5.Def_DecayRate} ``decay rate of $\D[\unk](t)$.'' The numerical simulation tool is put on \cite{Julia_Numerical_Simulation_Tool}.

\subsection{Effect of the largeness of the diffusion coefficient}
\label{sec;4.dcenter}
We observe a decay rate of $\D[\unk](t)$ due to the difference in the largeness of the diffusion coefficients. 
To do this, we fix the initial data as \eqref{eq:4.InitialData}, $d_{\text{coeff}}=0.5$, and $d_{\text{osc}}=4\pi$, that is,
\begin{equation}
    d(x) = d_{\text{center}}+0.5\sin(4\pi x).
\end{equation}
We simulate \eqref{eq:4.Forward_Euler_Scheme} numerically with the parameter $d_{\text{center}}=1.0$, $1.5$, $2.0$, $2.5$, and $3.0$ for $\mu_{\text{osc}}=2\pi$ and $5\pi$. 
Table \ref{tab:d_center_2pi} shows the relationship among $d_{\text{center}}$, the maximum of $|(\log d(x))_x|$, and the minimum decay rate of $\D[\unk](t)$ in terms of $\mu_{\text{osc}}=2\pi$. Table \ref{tab:d_center_5pi} shows the one in terms of $\mu_{\text{osc}}=5\pi$.

\begin{table}[htbp]
    \centering
    \begin{tabular}{ccc}\hline
        $\dcenter$ & the maximum of $|(\log d(x))_x|$ & decay rate of $\D[\unk](t)$ \\ \hline
         1.0 & 7.253 &  5.421 \\
         1.5 & 4.439 & 11.765 \\
         2.0 & 3.245 & 16.660 \\
         2.5 & 2.562  & 21.210 \\
         3.0 & 2.122 & 25.638 \\\hline
    \end{tabular}
    \caption{For the case $\mu_{\text{osc}}=2\pi$ and as the parameter $d_{\text{center}}$, we calculate the maximum of $|(\log d(x))_x|$ and decay rate of $\D[\unk](t)$.}
    \label{tab:d_center_2pi}
\end{table}

\begin{table}[htbp]
    \centering
    \begin{tabular}{ccc}\hline
        $\dcenter$ & the maximum of $|(\log d(x))_x|$ & decay rate of $\D[\unk](t)$ \\ \hline
         1.0 & 7.253 & 4.848 \\
         1.5 & 4.439 & 11.604 \\
         2.0 & 3.245 & 16.610 \\
         2.5 & 2.562  & 21.190 \\
         3.0 & 2.122 & 25.627 \\\hline
    \end{tabular}
    \caption{For the case $\mu_{\text{osc}}=5\pi$ and as the parameter $d_{\text{center}}$, we calculate the maximum of $|(\log d(x))_x|$ and decay rate of $\D[\unk](t)$.}
    \label{tab:d_center_5pi}
\end{table}

Since $(\log d(x))_x=d_x(x)/d(x)$, $|(\log d(x))_x|$ is decreasing in terms of $\dcenter$. When $\dcenter$ increases, the diffusion is stronger, and therefore the dissipation function $\D[\unk](t)$ decays faster.

\subsection{Effect of the coefficient of the oscillation function}
\label{sec;4.dcoeff}

We next observe the decay rate of $\D[\unk](t)$ due to the difference of the coefficient of the oscillation. 
To do this, we fix the initial data as \eqref{eq:4.InitialData}, $d_{\text{center}}=1.0$, and $d_{\text{osc}}=4\pi$, that is,
\begin{equation}
    d(x) = 1.0+d_{\text{coeff}}\sin(4\pi x).
\end{equation}
We simulate \eqref{eq:4.Forward_Euler_Scheme} numerically with the parameter $d_{\text{coeff}}=0.0,$ $0.1$, $0.2$, $0.3$, $0.4$, and $0.5$ for $\mu_{\text{osc}}=2\pi$ and $5\pi$. 
Table \ref{tab:d_coeff_2pi} shows the relationship among $d_{\text{center}}$, the maximum of $|(\log d(x))_x|$, and the minimum decay rate of $\D[\unk](t)$ in terms of $\mu_{\text{osc}}=2\pi$. Table \ref{tab:d_coeff_5pi} shows the one in terms of $\mu_{\text{osc}}=5\pi$.

\begin{table}[htbp]
    \centering
    \begin{tabular}{ccc}\hline
        $\dcoeff$ & the maximum of $|(\log d(x))_x|$ & decay rate of $\D[\unk](t)$ \\ \hline
         0.0 & 0.000 & 8.952 \\
         0.1 & 1.263 & 8.903 \\
         0.2 & 2.562 & 8.547 \\
         0.3 & 3.946 & 7.880 \\
         0.4 & 5.481 & 6.847 \\
         0.5 & 7.253 & 5.421 \\\hline
    \end{tabular}
    \caption{For the case $\mu_{\text{osc}}=2\pi$ and as the parameter $\dcoeff$, we calculate the maximum of $|(\log d(x))_x|$ and decay rate of $\D[\unk](t)$.}
    \label{tab:d_coeff_2pi}
\end{table}

\begin{table}[htbp]
    \centering
    \begin{tabular}{ccc}\hline
        $\dcoeff$ & the maximum of $|(\log d(x))_x|$ & decay rate of $\D[\unk](t)$ \\ \hline
         0.0 & 0.000 & 8.952 \\
         0.1 & 1.263 & 8.895 \\
         0.2 & 2.562 & 8.490 \\
         0.3 & 3.946 & 7.696 \\
         0.4 & 5.481 & 6.478 \\
         0.5 & 7.253 & 4.848 \\\hline
    \end{tabular}
    \caption{For the case $\mu_{\text{osc}}=5\pi$ and as the parameter $\dcoeff$, we calculate the maximum of $|(\log d(x))_x|$ and decay rate of $\D[\unk](t)$.}
    \label{tab:d_coeff_5pi}
\end{table}

Since $(\log d(x))_x=d_x(x)/d(x)$, $|(\log d(x))_x|$ is increasing in terms of $\dcoeff$. When $\dcoeff$ increases, the spatial inhomogeneity is stronger
 and therefore the dissipation function $\D[\unk](t)$ decays slower.

\subsection{Effect of the oscillation}
\label{sec;4.dcosc}

We next observe the decay rate of $\D[\unk](t)$ due to the difference in the oscillation $d_{\mathrm{osc}}$. 
To do this, we fix the initial data as \eqref{eq:4.InitialData}, $d_{\text{center}}=1.0$, and $dcoeff=0.5$, that is,
\begin{equation}
    d(x) = 1.0+0.5\sin(\dosc x).
\end{equation}
We simulate \eqref{eq:4.Forward_Euler_Scheme} numerically with the parameter $\dosc=0.0,$ $0.5\pi$, $1.0\pi$, $1.5\pi$, $2.0\pi$, $2.5\pi$, $3.0\pi$, $3.5\pi$, and $4.0\pi$ for $\mu_{\text{osc}}=2\pi$ and $5\pi$. 
Table \ref{tab:d_osc_2pi} shows the relationship among $d_{\text{center}}$, the maximum of $|(\log d(x))_x|$, and the minimum decay rate of $\D[\unk](t)$ in terms of $\mu_{\text{osc}}=2\pi$. Table \ref{tab:d_osc_5pi} shows the one in terms of $\mu_{\text{osc}}=5\pi$.

\begin{table}[htbp]
    \centering
    \begin{tabular}{ccc}
    \hline
    $d_{\mathrm{osc}}$ & the maximum of $|(\log d(x))_x|$ & Decay rate of $\D[\unknown](t)$ \\
    \hline
    $0.0\pi$   & 0.000 & 8.952 \\
    $0.5\pi$ & 0.785 & 11.034  \\
    $1.0\pi$   & 1.571 & 9.325  \\
    $1.5\pi$ & 2.720 & 6.246 \\
    $2.0\pi$ & 3.627 & 8.625   \\
    $2.5\pi$ & 4.532 & 12.416 \\
    $3.0\pi$   & 5.441 & 14.774  \\
    $3.5\pi$ & 6.343 & 6.926   \\
    $4.0\pi$   & 7.253 & 5.421  \\
    \hline
    \end{tabular}
    \caption{For the case $\mu_{\text{osc}}=2\pi$ and as the parameter $\dcoeff$, we calculate the maximum of $|(\log d(x))_x|$ and decay rate of $\D[\unk](t)$.}
    \label{tab:d_osc_2pi}
\end{table}

\begin{table}[htbp]
    \centering
    \begin{tabular}{ccc}
    \hline
    $d_{\mathrm{osc}}$ & the maximum of $|(\log d(x))_x|$ & Decay rate of $\D[\unknown](t)$ \\
    \hline
    $0.0\pi$   & 0.000 & 8.952\\
    $0.5\pi$ & 0.785 & 11.039 \\
    $1.0\pi$   & 1.571 & 9.402 \\
    $1.5\pi$ & 2.720 & 6.452\\
    $2.0\pi$ & 3.627 & 8.623 \\
    $2.5\pi$ & 4.532 & 12.639 \\
    $3.0\pi$   & 5.441 & 15.830 \\
    $3.5\pi$ & 6.343 & 6.520  \\
    $4.0\pi$   & 7.253 & 4.848 \\
    \hline
    \end{tabular}
    \caption{For the case $\mu_{\text{osc}}=5\pi$ and as the parameter $\dcoeff$, we calculate the maximum of $|(\log d(x))_x|$ and decay rate of $\D[\unk](t)$.}
    \label{tab:d_osc_5pi}
\end{table}

First, $|(\log d(x))_x|$ increases in terms of $\dosc$. Unlike \S \ref{sec;4.dcenter} and \ref{sec;4.dcoeff}, there is no correlation between
$\dosc$ and the decay rate of $\D[\unk](t)$. We can explain why the decay rate is faster in $\dosc=0.5\pi$ and $1.0\pi$ than in the homogeneous case, since $d$ is greater than in the homogeneous case. However, we do not know why the decay rate is the fastest in $\dosc=3.0\pi$ and the slowest in $\dosc=4.0\pi$. Following the result, the largeness of $|(\log d(x))_x|$  might not be essential for the decay rate of $\D[\unk](t)$.

\section{Conclusion}

In this paper, we derived the nonlinear Fokker-Planck equation \eqref{eq:Nonlinear-Fokker-Planck} with spatial inhomogeneous diffusion under the natural boundary condition from the perspective of the law of energy dissipation. Focusing on the potential function of the velocity vector field $\mu=d(x)\log\unk+\phi(x)$, we showed a priori estimates and gave a sufficient condition of exponential decay for the dissipation function $\D[\unk](t)$. Finally, applying the forward Euler scheme, we numerically observed the interaction between spatial inhomogeneity of the diffusion and the decay of the dissipation function. In particular, we found that the effect of the oscillation of the spatial inhomogeneity is not simple for the decay of the dissipation function. It is an interesting problem to analyze the influences of spatial heterogeneity for asymptotics of the dissipation function.

\section*{Acknowledgments}

The work of Masashi Mizuno was partially supported by JSPS KAKENHI Grant Numbers JP22K03376, JP23H00085, and JP25KK0052.

\bibliographystyle{amsplain}
\bibliography{references}

\end{document}